\documentclass[12pt]{amsart}
\usepackage[shortlabels]{enumitem}
\usepackage[margin=2.6cm]{geometry}
\usepackage{amssymb,amsopn,amscd,amsbsy}
\usepackage[usenames,dvipsnames]{xcolor}
\usepackage{graphicx}
\usepackage{xspace,titletoc}
\usepackage{algpseudocode}
\usepackage{mathrsfs}
\usepackage{multirow}

\usepackage[utf8]{inputenc}
\allowdisplaybreaks

\usepackage[usenames,dvipsnames]{xcolor}
\definecolor{darkblue}{rgb}{0.0, 0.0, 0.55}
\definecolor{bordeaux}{rgb}{0.34, 0.01, 0.1}

\usepackage[pagebackref,colorlinks,linkcolor=bordeaux,citecolor=darkblue,urlcolor=black,hypertexnames=true]{hyperref}

\newtheorem{theorem}{Theorem}[section]
\newtheorem{corollary}[theorem]{Corollary}
\newtheorem{lemma}[theorem]{Lemma}{\rm}
\newtheorem{proposition}[theorem]{Proposition}
\newtheorem{thmA}{Theorem}

\theoremstyle{definition}
\newtheorem{definition}[theorem]{Definition}{\rm}
\newtheorem{remark}[theorem]{Remark}
\newtheorem{example}[theorem]{Example}

\numberwithin{equation}{section}

\DeclareMathOperator{\Sym}{Sym}
\DeclareMathOperator{\diag}{diag}

\DeclareMathOperator{\II}{II}

\DeclareMathOperator{\tr}{tr}

\DeclareMathOperator{\rank}{rank}

\DeclareMathOperator{\cs}{cs}
\DeclareMathOperator{\re}{re}
\DeclareMathOperator{\im}{im}
\newif\ifcomment
\commentfalse
\commenttrue

\makeatletter
\def\pmb@@#1#2#3{\leavevmode\setboxz@h{#3}%
\dimen@-\wdz@
\kern-.5\ex@\copy\z@
\kern\dimen@\kern.25\ex@\raise.4\ex@\copy\z@
\kern\dimen@\kern.2\ex@\raise.3\ex@\copy\z@
\kern\dimen@\kern.15\ex@\raise.2\ex@\copy\z@
\kern\dimen@\kern.25\ex@\box\z@
}
\makeatother

\renewcommand{\subset}{\subseteq}

\begin{document}
\def\cA{\mathcal A}
\def\cH{\mathcal H}
\def\cK{\mathcal K}
\def\cU{\mathcal U}
\def\red{\color{red}}
\def\bl{\color{blue}}
\def\ora{\color{orange}}
\def\green{\color{green}}
\def\br{\color{brown}}
\newcommand\state{state\xspace}
\newcommand\ncstate{nc state\xspace}
\newcommand{\ncsostools}{\mathtt{NCSOStools}}
\newcommand{\ncpoltosdpa}{\mathtt{Ncpol2sdpa}}
\newcommand{\gloptipoly}{\mathtt{Gloptipoly}}
\newcommand{\twoone}{\II_1}
\def\la{\langle}
\def\ra{\rangle}
\def\e{{\rm e}}
\def\x{\mathbf{x}}
\def\by{\mathbf{y}}
\def\bz{\mathbf{z}}
\def\cC{\mathcal{C}}
\def\R{\mathbb{R}}
\def\O{\operatorname{O}}
\def\Mbb{\mathbb{M}}
\def\Sbb{\mathbb{S}}
\newcommand*{\sbb}[1]{\operatorname{S}_{#1}(\R)}
\def\T{\mathbb{T}}
\def\N{\mathbb{N}}
\def\K{\mathbb{K}}
\def\bK{\overline{\mathbf{K}}}
\def\Q{\mathbf{Q}}
\def\M{\mathbf{H}}
\def\O{\mathbf{O}}
\def\C{\mathbb{C}}
\def\Hbb{\mathbb{H}}
\def\P{\mathbf{P}}
\def\Z{\mathbb{Z}}
\def\A{\mathbf{A}}
\def\W{\mathbf{W}}
\def\bfone{\mathbf{1}}
\def\V{\mathbf{V}}
\def\AA{\overline{\mathbf{A}}}
\def\c{\mathbf{C}}
\def\bL{\mathbf{L}}
\def\bS{\mathbf{S}}
\def\Y{\mathbf{Y}}
\def\X{\mathbf{X}}
\def\G{\mathbf{G}}
\def\Bbb{\mathbb{B}}
\def\Dbb{\mathbb{D}}
\def\f{\mathbf{f}}
\def\z{\mathbf{z}}
\def\y{\mathbf{y}}
\def\d{\hat{d}}
\def\bx{\mathbf{x}}
\def\y{\mathbf{y}}
\def\h{\mathbf{h}}
\def\u{\mathbf{u}}
\def\g{\mathbf{g}}
\def\w{\mathbf{w}}
\def\cX{\mathcal{X}}
\def\a{\mathbf{a}}
\def\q{\mathbf{q}}
\def\u{\mathbf{u}}
\def\vb{\mathbf{v}}
\def\s{\mathcal{S}}
\def\cD{\mathcal{D}}
\def\co{{\rm co}\,}
\def\cp{{\rm CP}}
\def\fatT{\pmb{\mathscr{T}}}
\def\skinnyT{\mathscr{T}}
\def\fatS{\pmb{\mathscr{S}}}
\def\skinnyS{\mathscr{S}}
\def\tg{\tilde{f}}
\def\tx{\tilde{\x}}
\def\supmu{{\rm supp}\,\mu}
\def\supnu{{\rm supp}\,\nu}
\def\m{\mathcal{M}}
\def\bR{\mathbf{R}}
\def\om{\mathbf{\Omega}}
\def\c{\mathbf{c}}
\def\s{\mathcal{S}}
\def\k{\mathcal{K}}
\def\la{\langle}
\def\ra{\rangle}
\def\sig{\varsigma}

\newcommand{\bra}[1]{\mathinner{\langle #1|}}
\newcommand{\ket}[1]{\mathinner{|#1\rangle}}
\newcommand{\braket}[2]{\mathinner{\langle #1|#2\rangle}}
\newcommand{\dyad}[1]{| #1\rangle \langle #1|}

\def\cP{\mathcal{P}}
\def\cM{\mathcal{M}}
\def\QM{\operatorname{QM}}
\def\cQ{\operatorname{QM}_\sig}
\def\cN{\mathcal{N}}
\def\cF{\mathcal{F}}
\def\cE{\mathcal{E}}
\def\cB{\mathcal{B}}
\def\cS{\mathcal{S}}

\def\smileL{\overset{\smallsmile}{L}}
\def\blambda{{\boldsymbol{\lambda}}}
\def\bsigma{{\Delta}}
\def\RX{\R \langle \underline{x} \rangle}
\def\CX{\C \langle \underline{x} \rangle}
\def\RXk{\R \langle \underline{x}(I_k) \rangle}
\def\RXonetwo{\R \langle \underline{x}(I_1 \cap I_2) \rangle}
\def\CX{\C \langle \underline{x} \rangle}
\def\KX{\K \langle \underline{x} \rangle}
\def\uX{\underline X}
\def\uA{\underline A}
\def\uB{\underline B}
\def\ux{\underline x}
\def\mx{\langle\underline x\rangle}
\def\uY{\underline Y}
\def\uy{\underline y}
\def\SymS{\Sym \fatS}
\def\ov{\overline{o}}
\def\und{\underline{o}}
\newcommand{\victor}[1]{\Vi{#1}}
\newcommand{\victorshort}[1]{\todo[inline,color=purple!30]{VM: #1}}
\newcommand{\igor}[1]{\Ig{#1}}
\newcommand{\revision}[1]{{{\color{blue}#1}}}

\makeatletter
\providecommand{\leftsquigarrow}{%
  \mathrel{\mathpalette\reflect@squig\relax}%
}
\newcommand{\reflect@squig}[2]{%
  \reflectbox{$\m@th#1\rightsquigarrow$}%
}
\makeatother

\makeatletter
\newcommand{\mycontentsbox}{%
\printindex
{\centerline{NOT FOR PUBLICATION}
\addtolength{\parskip}{-2.0pt}\normalsize
\tableofcontents}}
\def\enddoc@text{\ifx\@empty\@translators \else\@settranslators\fi
\ifx\@empty\addresses \else\@setaddresses\fi
\newpage\mycontentsbox
}
\makeatother

\colorlet{commentcolour}{green!50!black}
\newcommand{\comment}[3]{%
\ifcomment%
	{\color{#1}\bfseries\sffamily(#3)%
	}%
	\marginpar{\textcolor{#1}{\hspace{3em}\bfseries\sffamily #2}}%
	\else%
	\fi%
}
\newcommand{\Ig}[1]{
	\comment{magenta}{I}{#1}
}
\newcommand{\Vi}[1]{
	\comment{blue}{V}{#1}
}
\newcommand{\Jie}[1]{
	\comment{green}{J}{#1}
}
\newcommand{\idea}[1]{\textcolor{red}{#1(?)}}

\newcommand{\Expl}[1]{
	{\tag*{\text{\small{\color{commentcolour}#1}}}%
	}
}

\renewcommand{\algorithmicrequire}{\textbf{Input:}}
\renewcommand{\algorithmicensure}{\textbf{Output:}}

\linespread{1.2}

\title[State polynomials]{State polynomials: positivity, optimization and nonlinear Bell inequalities}

\author{Igor Klep \and Victor Magron \and Jurij Vol\v{c}i\v{c} \and Jie Wang}

\address{Igor Klep: Faculty of Mathematics and Physics, Department of Mathematics,  University of Ljubljana \& Institute of Mathematics, Physics and Mechanics, Ljubljana, Slovenia}
\email{igor.klep@fmf.uni-lj.si}
\thanks{IK was supported by the 
Slovenian Research Agency program P1-0222 and grants 
J1-2453, N1-0217, J1-3004.}
\address{Victor Magron: LAAS-CNRS \& Institute of Mathematics from Toulouse, France}
\email{vmagron@laas.fr}
\thanks{VM was supported by the EPOQCS grant funded by the LabEx CIMI (ANR-11-LABX-0040), the FastQI grant funded by the Institut Quantique Occitan, the PHC Proteus grant
46195TA, the European Union’s Horizon 2020 research and innovation programme under the Marie Sk{\l}odowska-Curie Actions, grant agreement 813211 (POEMA), by the AI Interdisciplinary Institute ANITI funding, through the French ``Investing for the Future PIA3'' program under the Grant agreement n${}^\circ$ ANR-19-PI3A-0004 as well as by the National Research Foundation, Prime Minister’s Office, Singapore under its Campus for Research Excellence and Technological Enterprise (CREATE) programme.}
\address{Jurij Vol\v{c}i\v{c}: Department of Mathematics, Drexel University, Pennsylvania}
\email{jurij.volcic@drexel.edu}
\thanks{JV was supported by the NSF grant DMS-1954709.}
\address{Jie Wang: Academy of Mathematics and Systems Science, Chinese Academy of Sciences, Beijing, China}
\email{wangjie212@amss.ac.cn}
\thanks{JW was supported by the NSFC grant 12201618.}

\date{}

\begin{abstract}
This paper introduces state polynomials, i.e., polynomials in noncommuting variables and formal states of their products. A state analog of Artin's solution to Hilbert's 17th problem is proved showing that state polynomials, positive over all matrices and matricial states,
are sums of squares with denominators. Somewhat surprisingly, it is also established that a Krivine-Stengle Positivstellensatz fails to hold in the state polynomial setting. Further, archimedean Positivstellens\"atze in the spirit of Putinar and Helton-McCullough are presented leading to a hierarchy of semidefinite relaxations converging monotonically to the optimum of a state polynomial subject to state constraints. This hierarchy can be seen as a state analog of the Lasserre hierarchy for optimization of polynomials, and the Navascu\'es-Pironio-Ac\'in scheme for optimization of noncommutative polynomials. 
The motivation behind this theory arises from the study of correlations in quantum networks.
Determining the maximal quantum violation of a polynomial Bell inequality for an arbitrary network is reformulated as a state polynomial optimization problem. Several examples of quadratic Bell inequalities in the bipartite and the bilocal tripartite scenario are analyzed.
To reduce the size of the constructed SDPs, 
sparsity, sign symmetry and conditional expectation of the observables' group structure are exploited.
To obtain the above-mentioned results, techniques from noncommutative algebra, real algebraic geometry, operator theory, and convex optimization are employed.
\end{abstract}



\maketitle

\newpage

\section{Introduction}
\label{sec:intro}

This paper introduces the class of \emph{(noncommutative) state polynomials}, i.e., polynomials in noncommutative (nc)  variables, such as matrices or operators, and formal states of their products. 
Such polynomials are naturally evaluated 
over finite or infinite-dimensional Hilbert spaces $\cH$
by replacing
each variable by a bounded operator on $\cH$, and picking a state, i.e., a positive unital linear functional on the set of bounded operators $\cB(\cH)$.
The aim of the paper is to study positivity and optimization of state polynomials, and develop corresponding algebraic positivity certificates and associated algorithms.
The main motivation for studying state polynomials arises from quantum information theory, in particular nonlinear Bell inequalities \cite{Chaves,PHBB} for correlations in quantum networks \cite{Fri,pozas2019bounding,tavakoli22}. 
Namely, it turns out that computing the maximum quantum violation of a polynomial Bell inequality in the standard Bell scenario corresponds to optimizing a state polynomial under nc (in)equality constraints; that is, constraints only involve nc variables and not the state. 
For more general quantum networks, polynomial Bell inequalities correspond to state polynomial optimization problems subject to both nc and state (in)equalities. 

In the free nc context, i.e., in the absence of states, several representation results for positive polynomials (or Positivstellens\"atze) have been derived, allowing one to perform optimization.
One of the central results from Helton and McCullough independently \cite{Helton02,McCullSOS} asserts that all positive semidefinite polynomials are \textit{sums of hermitian squares} (SOHS). 
This in turn allows one to minimize the eigenvalue of an nc polynomial.
One can also minimize the eigenvalue of an nc polynomial subject to a finite number of nc polynomial inequality constraints, i.e., over a basic nc semialgebraic set. 
More precisely, a non-decreasing sequence of lower bounds of the minimal eigenvalue can be obtained, each bound corresponding to the solution of a semidefinite program (SDP)\footnote{That is, the optimum of a linear function subject to linear matrix inequality (LMI) constraints.}.
Thanks to the Helton-McCullough representation theorem \cite{Helton04}, the corresponding hierarchy of lower bounds converges to the minimal eigenvalue if the quadratic module generated by the polynomials describing the basic nc semialgebraic set is archimedean.
This framework is the nc variant of the nowadays famous Lasserre's hierarchy \cite{Las01sos} for \emph{commutative} polynomial optimization, based on the representation by Putinar \cite{Putinar1993positive} of positive polynomials over basic closed semialgebraic sets. 
Hierarchies of semidefinite programs have been applied and generalized to different nc optimization problems  \cite{Helton04,navascues2008convergent,pironio2010convergent,nctrace}. 
In the seminal paper \cite{navascues2008convergent},  Navascu\'es, Pironio and Ac\'in (NPA) provide such a hierarchy  to bound the maximal violation levels of linear Bell inequalities after casting the initial quantum information problem as an eigenvalue maximization problem;
cf.~\cite{doherty2008quantum}.
Extensions to trace minimization of nc polynomials have been derived in \cite{nctrace}.
More recently, several hierarchies have been derived in   \cite{Gribling19,gribling2022bounding} to provide lower bounds for various matrix factorization ranks. 
These hierarchies have been concretely implemented in the Matlab library {\tt NCSOStools} \cite{burgdorf16} and the Julia library {\tt TSSOS} \cite[Appendix~B]{sparsebook}.\looseness=-1

Recent efforts significantly extend these frameworks to the case of optimization problems involving \emph{trace} polynomials, i.e.,  polynomials in nc variables and traces of their products.
In \cite{klep2018positive}, the first and thirds authors focused on trace polynomials being positive on semialgebraic sets of \emph{fixed size} matrices, and derived several Positivstellens\"atze, 
including
a Putinar-type Positivstellensatz stating that any positive trace polynomial admits a weighted SOHS decomposition without denominators.
In \cite{KMV}, the first, second and third authors generalized the above framework to the free setting,  by providing a Putinar-type Positivstellensatz for trace polynomials which are positive on tracial semialgebraic sets, where the evaluations are performed on von Neumann algebras.
This latter framework was applied in \cite{huber2022dimension} to detect entanglement of Werner state witnesses in a dimension-free way. 
In the univariate case, a tracial analog of Artin's solution to Hilbert's 17th problem was provided in \cite{KPV}, where it is proved that a positive semidefinite univariate trace polynomial is a quotient of sums of products of squares and traces of squares of trace polynomials. 
In the multivariate unconstrained setting, it is shown in \cite{KSV} that trace-positive nc polynomials can be 
``weakly'' approximated by SOHS and commutators of regular nc rational functions.

From the point of view of quantum information, the trace polynomial optimization framework from \cite{KMV} allows us to obtain bounds on violation levels of nonlinear Bell inequalities corresponding to \emph{maximally} entangled states. 
In this paper we rely on state polynomial optimization that is less restrictive, as it can provide violation bounds reached by (not necessary maximally) entangled states. 
From the point of view of operator theory, there is a correspondence between states on a Hilbert space $\cH$ and trace-class operators on $\cH$, but the reformulation of a state polynomial optimization problem into one with trace polynomials involves the \emph{non-normalized trace}, in which case there is no dimension-independent theory of positivity, 
necessitating the introduction of this new class of objects, i.e., (nc) state polynomials.

\subsection*{Contributions and main results}
A \emph{state polynomial} in nc variables $x_1,\dots,x_n$ is a real polynomial in formal state symbols $\sig(w)$, where $w$ is a word in $x_1,\dots,x_n$. 
More generally, an \emph{\ncstate polynomial} is a polynomial in $x_1,\dots,x_n$ and formal states of their words. 
For example, $f=\sig(x_1x_2x_1)-\sig(x_1)\sig(x_1x_2)$ is a state polynomial, 
and $h=\sig(x_1^2)x_2x_1+\sig(x_1)\sig(x_2x_1x_2)$ is an \ncstate polynomial. At a pair of bounded operators $\uX=(X_1,X_2)$
on Hilbert space $\cH$
 and a state $\lambda$ on $\cB(\cH)$, they are evaluated as $f(\lambda;\uX)=\lambda(X_1X_2X_1)-\lambda(X_1)\lambda(X_1X_2)$ and
$h(\lambda;\uX)=\lambda(X_1^2)X_2X_1+\lambda(X_1)\lambda(X_2X_1X_2)I$.

State polynomials form a commutative algebra denoted $\skinnyS$, and nc state polynomials form a  noncommutative algebra denoted $\fatS$. There is a canonical involution $\star$ on $\fatS$ that fixes $\skinnyS\cup\{x_1,\ldots,x_n\}$ element-wise, and an $\skinnyS$-linear map $\sig:\fatS\to\skinnyS$.

After establishing the algebraic framework for state polynomials in Section \ref{sec:prelim} and their function theoretic perspective in Section \ref{sec:freean}, we prove our first main result, the affirmative answer to a state polynomial analog of Hilbert's 17th problem from real algebraic geometry \cite{marshallbook,sche}.

\begin{thmA}[{Theorem \ref{t:h17}}]\label{thm:a}
Let $f$ be a state polynomial. 
Then $f(\lambda;\uX)\ge0$ for all matricial states $\lambda$ and tuples of symmetric matrices $\uX$ if and only if $f$ is a quotient of sums of products of elements of the form $\sig(hh^\star)$ for an nc state polynomial $h$.
\end{thmA}

For example, 
$$\sig(x_1^2)\sig(x_2^2)-\sig(x_1x_2)^2=
\frac{\sig\Big(
\big(\sig(x_1^2)x_2-\sig(x_1x_2)x_1\big)^2
\Big)}{\sig(x_1^2)}$$
is an algebraic certificate for the Cauchy-Schwarz inequality, 
a sum of hermitian squares (SOHS) certificate
of the form guaranteed for all global state polynomial inequalities by Theorem \ref{thm:a}.
As a consequence of Theorem \ref{thm:a}, positivity of a state polynomial on all matrix tuples and matricial states implies positivity on all bounded operators and states.

After global positivity, we turn to constrained positivity of state polynomials in Section \ref{sec:noncyclic}. We restrict ourselves to constraint sets $C\subset\fatS$ that are \emph{balanced}: namely, $C$ is closed under the involution $\star$, and the non-symmetric elements of $C$ come in pairs with their negatives (to allow us to handle equality constraints).
Let $\cH$ be a separable real Hilbert space. Given a balanced set $C\subset\fatS$, let $\cD_C^\infty$ be the set of all pairs $(\lambda;\uX)$ of a state $\lambda$ on $\cB(\cH)$ and $\uX\in\cB(\cH)^n$ such that $c(\lambda;\uX)\succeq0$ for all $c\in C$. We call $\cD_C^\infty$ the \emph{state semialgebraic set} constrained by $C$.
While Theorem \ref{thm:a} gives an SOHS certificate for global positivity of a state polynomial (even when only matrix evaluations are considered), there is no comparable analog for positivity on arbitrary state semialgebraic sets. 
In Section \ref{ssec:nogo} we show that the state versions of some of the classic (Krivine-Stengle and Schm\"udgen) Positivstellens\"atze fail in general.
Nevertheless, 
there is an analog of Putinar's archimedean Positivstellensatz \cite{Putinar1993positive}. 
We say that $C\subset\fatS$ is \emph{algebraically bounded} if $N-x_1^2-\cdots-x_n^2 = \sum_ip_ic_ip_i^\star$ for some $c_i\in C$ and nc polynomials $p_i$. Note that $\cD_C^\infty$ for an algebraically bounded $C$ is bounded in operator norm; conversely, if $\cD_C^\infty$ is bounded, then one can make $C$ algebraically bounded 
without changing the state semialgebraic set $\cD_C^\infty$
by adding a single constraint.
For algebraically-bounded constraint sets, we obtain the following Positivstellensatz.

\begin{thmA}[{Theorem \ref{thm:nocyc}}]\label{t:b}
Let $f$ be a state polynomial, and $C$ a balanced algebraically bounded set of nc state polynomials. 
Then $f\ge0$ on $\cD_C^\infty$ if and only if for every $\varepsilon>0$,
$$f+\varepsilon = \sum_i \sig(h_ic_ih_i^\star)$$
for some nc state polynomials $h_i$ and $c_i\in\{1\}\cup C$.
\end{thmA}

Theorem \ref{t:b} is the cornerstone of the state optimization framework we develop in Section \ref{sec:hierarchy}.
For a state polynomial $f$ and a balanced algebraically bounded set $C \subset \fatS$, Theorem \ref{t:b} gives rise to an SDP hierarchy that produces a convergent increasing sequence with limit $\inf_{\cD_C^\infty} f$ (Corollary \ref{cor:pure_cvg}).
Under a mild condition on $C$, these SDPs satisfy strong duality (Proposition \ref{p:nogap}). Furthermore, under flatness and extremal assumptions, the dual SDPs and a variant of the Gelfand-Naimark-Segal construction allow us to extract a finite-dimensional minimizer for 
$\inf_{\cD_C^\infty} f$, and thus obtain finite convergence of our hierarchy (Proposition \ref{prop:st_flat}).
The complexity of the involved SDPs grows rather quickly;
nevertheless, we can
exploit
correlative sparsity (Theorem \ref{t:sparse}) and sign symmetry (Theorem \ref{t:sign}) patterns in $f$ and $C$ to reduce the SDP sizes considerably.

Finally, we apply our newly developed theory to quantum correlations in networks.
Section \ref{sec:bell} considers nonlinear Bell inequalities \cite{Uffink,Chaves} in the standard Bell scenario, where two parties share an entanglement source. 
While a linear Bell inequality for (reduced) quantum correlations in such a scenario corresponds to eigenvalue optimization of an nc polynomial,
a polynomial Bell inequality for quantum correlations corresponds to a state polynomial optimization problem. 
The form of the constraints arising from the quantum mechanical formalism allows for a further reduction of the size of the 
obtained SDPs in the hierarchy
using a conditional expectation induced by the underlying group structure of binary observables (Proposition \ref{p:condexp}).
Section \ref{sec:networks} generalizes the aforementioned correspondence to correlation inequalities in general network scenarios \cite{Fri}. That is, several entanglement sources and sharing patterns are permitted. 
Following \cite{pozas2019bounding,ligthart21,renou22,ligthart23}, reduced quantum models for a network can be characterized using state polynomial constraints. This allows us to apply our optimization results to analyze polynomial Bell inequalities for correlations in arbitrary networks.

\begin{thmA}[{Corollaries \ref{cor:pure_cvg} and \ref{c:network}}]
The largest quantum violation of a polynomial Bell inequality for classical correlations in a network scenario is the limit 
of a convergent decreasing sequence produced by the
Positivstellensatz-generated SDP hierarchy.
\end{thmA}

Using the derived optimization tools, we establish novel largest quantum violations or their nontrivial upper bounds for various polynomial Bell inequalities in the bipartite scenario (Section \ref{sec:exa}) and in the bilocal tripartite scenario (Section \ref{sec:biloc}) from the literature.

\subsection*{Acknowledgements}
The authors thank Timotej Hrga for sharing his expertise and solving the large semidefinite program pertaining to one of the Bell inequality examples.

\section{Preliminaries}
\label{sec:prelim}

We begin by recalling basic notions about noncommutative polynomials, introducing \state polynomials and corresponding semialgebraic sets that will be used throughout the paper.

\subsection{Noncommutative polynomials and \state polynomials}
\label{sec:prelim_nc}
Let $\sbb{k}$ denote the space of all real symmetric matrices of order $k$. For a set $A$, we use $|A|$ to denote its cardinality.
For a fixed $n \in \N$, we consider a finite alphabet $x_1,\dots,x_n$ and generate all possible words of finite length in these letters. 
The empty word is denoted by 1.
The resulting set of words is the {\em free monoid} $\mx$, with $\underline{x} = (x_1,\dots, x_n)$. 
Let $|w|$ denote the length of $w\in\mx$.
We denote by $\RX$ the set of real polynomials in noncommutative variables, abbreviated as {\em nc polynomials}.
The free algebra $\RX$ is equipped with the involution $\star$ that fixes $\R \cup \{x_1,\dots,x_n\}$ point-wise and reverses words, so that $\RX$ is the $\star$-algebra freely generated by $n$ symmetric variables $x_1,\dots,x_n$. 

For $w\in \mx\setminus\{1\}$ let $\sig(w)$ be a symbol subject to the relation $\sig(w)=\sig(w^\star)$, and let
$$\skinnyS := \R\big[\sig(w)\colon  w\in \mx\setminus\{1\}\big],$$
a commutative polynomial ring in infinitely many variables. An element in $\skinnyS$ of the form  $\prod_{j=1}^m \sig(u_j)$ for $u_j\in\mx\setminus\{1\}$ is called an \emph{$\skinnyS$-word}. The set of all $\skinnyS$-words is a vector space basis of $\skinnyS$.
The \emph{degree} of an $\skinnyS$-word $\prod_j\sig(u_j)$ equals $\sum_j|u_j|$. The vector of $\skinnyS$-words whose degrees are no greater than $d$ is denoted by $W_d^{\skinnyS}$.

We also let $\fatS := \skinnyS \otimes \RX$ be the
free $\skinnyS$-algebra on $\ux$.
Elements of $\skinnyS$ are called \emph{\state polynomials}, and elements
of $\fatS$ are \emph{\ncstate polynomials}. For example, $\sig(x_1x_2)-\sig(x_1)\sig(x_2) \in \skinnyS$ and $x_1x_2x_1-\sig(x_2) x_1+\sig(x_1x_2)-\sig(x_1)\sig(x_2) \in \fatS = \skinnyS\langle x_1,x_2 \rangle$.
An \ncstate polynomial of the form  $\prod_{j=1}^m \sig(u_j) v$ for $u_j\in\mx\setminus\{1\}$ and $v\in\mx$ is called an \emph{$\fatS$-word}. The set of all $\fatS$-words is a vector space basis of $\fatS$.
The \emph{degree} of an $\fatS$-word $\prod_j\sig(u_j) v$ equals $|v|+\sum_j|u_j|$.
The \emph{degree} of $f \in \fatS$ is the maximal degree of $\fatS$-words in the expansion of $f$.\looseness=-1

The involution on $\fatS$, denoted also by $\star$, fixes $\{ x_1,\dots, x_n\}\cup \skinnyS$ point-wise, and reverses words from $\mx$.
The set of all {\em symmetric elements} of $\fatS$ is defined as $\SymS := \{f \in \fatS : f = f^\star  \}$.
We also consider the 
unital $\skinnyS$-linear $\star$-map $\sig:\fatS \to \skinnyS$ uniquely determined by $w\mapsto \sig(w)$ for $w\in\mx\setminus\{1\}$.

\subsection{State semialgebraic sets}
\label{sec:semialg}

First we recall some classical notions from functional analysis and operator algebras \cite{reed80,Tak02}.
Let $\cH$ be a separable real Hilbert space, i.e., $\cH$ admits a countable orthonormal basis, and let $\cB(\cH)$ be the Banach algebra of bounded linear operators on $\cH$.
A \emph{state} $\lambda$ on $\cB(\cH)$ \cite[Definition I.9.4]{Tak02} is a positive unital $\star$-linear functional, namely 
\begin{enumerate}[(a)]
\item $\lambda:\cB(\cH)\to\R$ is a linear map, \item $\lambda(I)=1$,
\item $\lambda(Y^*)=\lambda(Y)$ for all $Y\in\cB(\cH)$,
\item $\lambda(YY^*)\ge0$ for all $Y\in\cB(\cH)$.
\end{enumerate}
Notice that (c) is a consequence of (a) and (d), and $\lambda$ satisfying (a)-(d) is continuous in the norm topology \cite[Lemma I.9.9]{Tak02}.
Every unit vector $v\in\cH$ determines a \emph{vector state} $Y\mapsto \langle Yv,v\rangle$. 
More general states are obtained using \emph{trace-class} operators \cite[Section VI.6]{reed80}. If $(e_j)_j$ is an orthonormal basis of $\cH$, then $T\in\cB(\cH)$ is trace-class if the series $\sum_j\langle \sqrt{T^*T} e_j,e_j\rangle$ converges. Every positive semidefinite trace-class operator $\rho\in\cB(\cH)$ with trace 1 gives rise to a state via $Y\mapsto \tr(\rho Y)=\sum_j\langle \rho Y e_j,e_j\rangle$. Such $\rho$ is also called a \emph{density operator}, and the state it determines is called \emph{normal} \cite[Definition III.2.13]{Tak02}. If $\cH$ is finite-dimensional, then every state on $\cB(\cH)$ is normal.
Let $\cS(\cH)$ denote the set of all states on $\cB(\cH)$.

Given an \ncstate polynomial $a\in\fatS$, a state $\lambda\in\cS(\cH)$, and a tuple $\uX=(X_1,\dots,X_n)$ of self-adjoint operators $X_j=X_j^*\in\cB(\cH)$, there is a natural evaluation
$$a(\lambda;\uX)\in\cB(\cH)$$
obtained by replacing $w$ with $w(X)\in\cB(\cH)$ and $\sig(w)$ with $\lambda(w(X))\in\R$.
Equivalently, each pair $(\lambda;\uX)$ gives rise to a $\star$-representation $\fatS\to\cB(\cH)$ that intertwines $\sig:\fatS\to\skinnyS$ and $\lambda:\cB(\cH)\to \R$.

Throughout the paper let $\cH$ be a separable infinite-dimensional real Hilbert space; note that up to isomorphism, there is only one such Hilbert space, so the reader may have in mind the space of real square-summable-sequences $\cH=\ell^2$.

\def\vD{\vec{\mathcal{D}}}
\begin{definition}
A set $C\subseteq\fatS$ is \emph{balanced} if $C^\star=C$ and $-(C\setminus\SymS)\subseteq C$.
Given a balanced $C$ let
\begin{align}\label{eq:DSinf}
\mathcal{D}_C^\infty& := \{ (\lambda,\underline{X}) \in \cS(\cH)\times\cB(\cH)^n : 
X_j=X_j^*,\ c(\lambda;\underline{X}) \succeq 0 \ \text{for all}\ c\in C  \} 
\end{align}
and
\begin{align}
\label{eq:DS}
\mathcal{D}_C := \bigcup_{k \in \N} \{ (\lambda,\underline{X}) \in \cS(\R^k)\times\sbb{k}^n : c(\lambda;\underline{X}) \succeq 0 \ \text{for all}\ c\in C  \} \,.
\end{align}
Let $\vD_C^\infty$ and $\vD_C$ be the analogs of \eqref{eq:DSinf} and \eqref{eq:DS}, respectively, where one restricts only to vector states.
\end{definition}

Note that any subset of $\SymS$ is an example of a balanced set. The motivation behind more general balanced sets is to allow for non-symmetric {\it equality} constraints, for example commutation relations. Indeed, every non-symmetric element $c$ in a balanced set contributes inequalities $c(\lambda;\underline{X}) \succeq 0$ and $-c(\lambda;\underline{X}) \succeq 0$, and thus $c(\lambda;\underline{X}) =0$.

While \eqref{eq:DS} is a desired candidate for testing positivity of \state polynomials, one needs to consider bounded operators on an infinite-dimensional Hilbert space in order to obtain sums-of-squares certificates valid for sufficiently general $C$ (cf. \cite{Helton04} for examples in the freely noncommutative setting, or the refutation of Connes' embedding conjecture in the tracial setting \cite{CECfalse,CECsohs}).
Thus we mostly consider positivity on sets \eqref{eq:DSinf}.

One might wonder why we restrict ourselves to real Hilbert spaces. In the complex framework, the only difference is that the symbol $\sig(w)$ needs to be split into two symbols, the real and imaginary part, to properly define $\sig(w^\star)=\overline{\sig(w)}$. Thus one is pressed to work with real variables also in the complex framework. The real framework is also more convenient to work with in optimization, especially from the perspective of implementation using the standard semidefinite programming solvers. The complex adaptation is detailed in Section \ref{ss:cx}.

\subsection{State polynomials versus trace polynomials}\label{ss:SvT}

Before diving into the main contributions of this paper, let us discuss the relation between \ncstate polynomials and nc trace polynomials \cite{P76,KS17}, whose optimization perspective has been considered earlier by the first three authors \cite{KMV}.

Trace polynomials originated in invariant theory \cite{P76} to describe polynomial functions $\sbb{k}^n\to\R$ and polynomial maps $\sbb{k}^n\to\sbb{k}$ that are invariant and equivariant, respectively, under the orthogonal group acting on $\sbb{k}^n$ via simultaneous conjugation. To each $w\in\mx\setminus\{1\}$ ones assigns a formal trace symbol $\tr(w)$. These symbols are required to commute and to satisfy $\tr(w)=\tr(w^\star)$ and $\tr(uv)=\tr(vu)$ for $u,v,w\in\mx$. As in the case of (nc) state polynomials, one defines \emph{(pure) trace polynomials} as $\skinnyT=\R[\tr(w)\colon w\in\mx\setminus\{1\}]$ and \emph{nc trace polynomials} as $\fatT=\skinnyT\otimes\RX$. 
For example, trace polynomials $\tr(x_1)\tr(x_2x_1x_2)$ and $\tr(x_1)\tr(x_2^2x_1)$ are the same, but analogous state polynomials $\sig(x_1)\sig(x_2x_1x_2)$ and $\sig(x_1)\sig(x_2^2x_1)$ are not.
We endow $\fatT$ with the involution $\star$ that fixes the elements of $\{x_1,\dots,x_n\}\cup\skinnyT$, and with the degree function that is defined analogously as in the nc state polynomial case (namely, the degree of $\prod_j\sig(u_j)v$ equals $|v|+\sum_j|u_j|$).
There is also a natural unital $\skinnyT$-linear map $\tr:\fatT\to\skinnyT$, which satisfies $\tr(fg)=\tr(gf)$ for $f,g\in\fatT$.

Algebraically speaking, $\star$-algebras $\skinnyT$ and $\fatT$ are natural quotients of algebras $\skinnyS$ and $\fatS$, respectively, since the symbol $\tr$ satisfies the relations of the symbol $\sig$. Note however that the symbol $\tr$ also satisfies $\tr(uv)=\tr(vu)$, which is a property not shared among all states. In \cite{KMV}, the authors therefore considered evaluations of nc trace polynomials on self-adjoint operators from real \emph{von Neumann algebras}, and \emph{tracial states}. A real von Neumann algebra $\cF$ is $*$-subalgebra of bounded operators on a real Hilbert space, closed in the weak operator topology. A tracial state on $\cF$ is a positive unital $\star$-linear functional $\lambda:\cF\to\R$ that satisfies $\lambda(ab)=\lambda(ba)$ for $a,b\in\cF$. For operators on a $k$-dimensional Hilbert space, there is only one tracial state: the normalized trace on $k\times k$ real matrices. On the other hand, there is no tracial state on $\cB(\cH)$ for an infinite-dimensional Hilbert space $\cH$, which is the reason why von Neumann algebras need to be considered.
The discussed restriction to tracial states and von Neumann algebras distinguishes positivity of trace polynomials and positivity of state polynomials. For example, the trace polynomial $t=\tr(x_1^6)\tr(x_2^6)^2-\tr(x_1^2x_2^4)^3$ is nonnegative on all tracial states and self-adjoint operators from any von Neumann algebra, by H\"older's inequality for noncommutative $L^p$ spaces \cite[Theorem IX.2.13]{Tak03}. On the other hand, the state polynomial $s=\sig(x_1^6)\sig(x_2^6)^2-\sig(x_1^2x_2^4)^3$ has value $-18$ at a pair of $2\times 2$ matrices $X_1,X_2$ and a state $\lambda$ on $2\times 2$ matrices,
$$X_1=\begin{pmatrix}1&1\\1&0
\end{pmatrix},\quad 
X_2=\begin{pmatrix}0&1\\1&1
\end{pmatrix},\quad
\lambda(Y)=\begin{pmatrix}1&0\end{pmatrix}
Y\begin{pmatrix}1\\0\end{pmatrix}.$$

The paper \cite{KMV} provided positivity certificates and optimization procedures for trace polynomials subject to tracial constraints. These results are of merit to noncommutative probability (since tracial states are the noncommutative analogs of expectations, and pure trace polynomials thus correspond to higher noncommutative moments) and certain studies in quantum information theory. For example, trace polynomial optimization can be used for finding violations of bipartite polynomial Bell inequalities attained by maximally entangled states (which correspond to tracial states), and entanglement detection in highly invariant multipartite Werner states \cite{huber2022dimension}. 

With the above review of trace polynomials in mind, let us list the main differences between state and trace polynomials, from the perspective of positivity and optimization. 
Firstly, trace polynomials and their function-theoretic properties date back to the work of Procesi \cite{P76}; on the other hand, state polynomials are completely new objects, and their function-theoretic properties are derived in Section \ref{sec:freean}.
Secondly, while the natural analog of Hilbert's 17h problem holds for state polynomials (Theorem \ref{t:h17}, its direct analog for trace polynomials fails (see the discussion at the beginning of Section \ref{sec:H17}). Thirdly, trace polynomial framework does not apply to the optimization problems in several operator variables where one wishes to quantify the optimal value over \emph{all} states. Concrete examples of such optimization problems are bipartite polynomial Bell inequalities over arbitrary (non-maximally) entangled states, and network Bell inequalities (considered in Sections \ref{sec:bell} and \ref{sec:biloc}, respectively).
Fourthly, the distinction between trace and state polynomials does not affect only the formal scheme of positivity certificates and optimization algorithms, but also the details in their derivation. For example, certificates for trace positivity are simpler than certificates for state positivity because tracial states satisfy the ``bounded marginal moments imply bounded mixed moments'' property, which fails for general states (cf. Remark \ref{r:nontr}). Similarly, that derivation of strong duality in state polynomial optimization differs from that of strong duality in trace polynomial optimization (cf. Proposition \ref{p:nogap}). Furthermore, while the setup of \cite{KMV} allows for a characterization of positive \emph{nc} trace polynomials on bounded domains \cite[Corollary 4.8]{KMV}, one obtains only characterization of positive state polynomials (but not nc state polynomials) on bounded domains (Theorem \ref{thm:nocyc}).

Finally, the correspondence between density trace-class operators (or density matrices, in the finite-dimensional case) and states mentioned earlier begs the question why one cannot derive positivity and optimization of state polynomials more directly from the corresponding already-established theory for trace polynomials. The reason is that optimization of trace polynomials pertains to evaluations of the {\it normalized} trace on matrices, or tracial states on von Neumann algebras. On the other hand, modelling state polynomials with trace polynomials (in an additional variable corresponding to the trace-class operator) would require considering the (usual) non-normalized trace. However, for evaluations of trace polynomials with respect to a non-normalized trace, there is no comparable dimension-independent theory of positivity and optimization.

\section{A functional perspective on \ncstate polynomials}\label{sec:freean}

One can draw a comparison between \ncstate polynomials and noncommutative functions as developed e.g.~in \cite{KVV}. 
While \ncstate polynomials are not noncommutative functions (since their evaluations on matrix tuples are not compatible with direct sums of matrix tuples), they nevertheless admit a closely related intrinsic characterization (Proposition \ref{p:freean}), and like nc polynomials, they are determined on matrices of bounded size (Proposition \ref{prop:id}). In the proofs of these statements we utilize evaluations of trace polynomials (Section \ref{ss:SvT}) on matrices using the usual (non-normalized) matrix trace.

\begin{proposition}\label{prop:id}
If $f \in \fatS$ is of degree $d$, then there exist
$\lambda\in\cS(\R^{2d+1})$ and $\uX\in\sbb{2d+1}^n$ such that $f(\lambda;\uX)\neq0$.
\end{proposition}

\begin{proof}
Let
$$f=\sum_{i=1}^k\alpha_i\prod_{j=1}^{m_i}\sig(u_{i,j}) v_i.$$
Define a trace polynomial in variables $x_0,\dots,x_n$
$$g=\sum_{i=1}^k \alpha_i (\tr x_0)^{d-m_i} \prod_{j=1}^{m_i} \tr(x_0 u_{i,j}) v_i.$$
Note that $g\neq0$ since $f\neq0$, and the degree of $g$ (as a trace polynomial) is $2d$. By \cite[Proposition 8.3]{P76}, $g$ is not constantly zero on $\sbb{2d+1}^{n+1}$. Since positive definite matrices are Zariski dense in symmetric matrices, there exist $(P,\uX)\in\sbb{2d+1}\times\sbb{2d+1}^n$ with $P\succ0$ such that $g(P,\uX)\neq0$.
By the construction of $g$ we have $g(\alpha P,\uX)=\alpha^d g(P,\uX)$ for all $\alpha\in\R$. Therefore
$$f(\lambda;\uX)=g\left(\tfrac{1}{\tr P}P,\uX\right)\neq0$$
where $\lambda\in\cS(\R^{2d+1})$ is given by $\lambda(Y)=\tr(\frac{1}{\tr P} P Y)$.
\end{proof}

As mentioned in Section \ref{sec:semialg}, states in $\cS(\R^k)$ are in one-to-one correspondence with $k\times k$ density matrices (positive semidefinite matrices in $\sbb{k}$ with trace 1); namely, a density matrix $\rho$ gives rise to a state 
$Y\mapsto \tr(\rho Y)$.
For the purpose of the next proposition we resort to this identification. Hence we view $\cS(\R^k)$ as a Zariski dense subset of the affine space of symmetric matrices with trace 1. By a polynomial function on $\cS(\R^k)\times \sbb{k}^n$ we therefore refer to a polynomial function on the corresponding real affine space of dimension $\frac{k(k+1)}{2}-1+n\frac{k(k+1)}{2}$. Furthermore, $\cS(\R^k)\times \sbb{k}^n$ inherits the diagonal conjugate action of the orthogonal group $\O_k$ 
on tuples of symmetric $k\times k$ matrices:
$$O(X_0,\dots,X_n)O^* =
(OX_0O^*,\dots,OX_nO^*).
$$
The following is an \ncstate analog of \cite[Theorem 6.1]{KVV} and \cite[Proposition 3.1]{KS17} for nc polynomials.
\begin{proposition}\label{p:freean}
A sequence $(f_k)_{k\in\N}$ of polynomial maps $f_k:\cS(\R^k)\times \sbb{k}^n\to\sbb{k}$ satisfies
\begin{enumerate}[\rm (a)]
    \item $f_k(O (\lambda;\uX)O^*)
    =Of_k(\lambda;\uX)O^*$ 
    for all $k\in\N$, $O\in \O_k$ and $(\lambda;\uX)\in \cS(\R^k)\times \sbb{k}^n$,
    \item $f_{\ell k}(\rho\otimes\lambda;\uX^{\oplus \ell})=f_k(\lambda;\uX)^{\oplus \ell}$ 
    for all $k,\ell\in\N$, $\rho\in \cS(\R^\ell)$ and $(\lambda;\uX)\in \cS(\R^k)\times \sbb{k}^n$,
    \item $\sup_k\deg f_k<\infty$,
\end{enumerate}
if and only if it is given by an \ncstate polynomial.
\end{proposition}
\begin{proof}
$(\Leftarrow)$: Let $w\in\mx$. Then
\begin{align*}
w(O\uX O^*)&=Ow(\uX)O^*,\\
\sig(w)(O(\lambda;\uX) O^*)&= \tr(O\lambda O^*\cdot O w(\uX)O^*)=\tr(\lambda\cdot w(\uX))=\sig(w)(\lambda;\uX),\\
w(\uX^{\oplus\ell})&=w(\uX)^{\oplus \ell},\\
\sig(w)(\rho\otimes\lambda;\uX^{\oplus\ell})&= \tr(\rho\otimes\lambda\cdot w(\uX)^{\oplus \ell})=\tr(\rho)\tr(\lambda\cdot w(\uX))=\sig(w)(\lambda;\uX)
\end{align*}
for all orthogonal matrices $O\in\O_k$, tuples $\uX\in\sbb{k}^n$, and positive semidefinite matrices $\lambda\in \sbb{k}$ and $\rho\in\sbb{\ell}$ of trace 1. Next, 
$w(\uX)$ is, as a matrix of polynomial expressions in the entries of $\uX$, of degree $d$. Similarly, $\sig(w)(\lambda;\uX)$ is, as a polynomial expression in the entries of $\lambda$ and $\uX$, of degree $d+1$.
Since the algebra $\fatS$ is generated by elements of the form $w$ and $\sig(w)$, it follows that every \ncstate polynomial satisfies (a)-(c).

$(\Rightarrow)$:
By (a) and \cite[Theorem 7.3]{P76}, for every $k\in\N$ there exists an nc trace polynomial $T_k$ in variables $x_0,\dots,x_n$ 
such that $T_k$ agrees with $f_k$ on $\cS(\R^k)\times\sbb{k}^n$, 
the expression $\tr(x_0)$ does not appear in $T_k$ (cf. \cite[Section 5]{P76}), and $\deg T_k=\deg f_k$.
Denote $d=\max_k\deg f_k$. 

For $k\ge d+1$, the nc trace polynomial $T_k$ defined as above is unique \cite[Proposition 8.3]{P76}. Let
\begin{equation}\label{e:expand}
T_k=\sum_i\alpha_{k,i}\prod_{j}\tr(u_{i,j}) v_i
\end{equation}
where $\prod_{j}\tr(u_{i,j}) v_i$ are distinct trace words with $|v_i|+\sum_j|u_{i,j}|\le d$.
By uniqueness and comparison of (b) for $\rho=\frac12 I_2$ and $\rho=\frac12(\begin{smallmatrix}1&1\\1&1\end{smallmatrix})$, the variable $x_0$ cannot appear in any $v_i$.
Furthermore, if there are $m$ occurrences of $x_0$ in $u_{i,j}$, then
$$\tr\Big(
u_{i,j}\big(\rho\otimes \lambda; I\otimes X_1,\dots,I\otimes X_n\big)
\Big) = 
\tr (\rho^m)\cdot \tr\big(
u_{i,j}(\lambda;X_1,\dots,X_n)
\big).
$$
Therefore, uniqueness of $T_k$ and comparison of (b) for $\rho=\frac12I_2$ and $\rho=\frac14(\begin{smallmatrix}1&0\\0&3\end{smallmatrix})$ imply that the variable $x_0$ can appear in each $u_{i,j}$ at most once.
Finally, uniqueness and comparison of (b) for $\rho=\frac12I_2$ and $\rho=(\begin{smallmatrix}1&0\\0&0\end{smallmatrix})$ show that $x_0$ appears in each $u_{i,j}$.
Thus \eqref{e:expand} becomes
\begin{equation}\label{e:expand1}
T_k=\sum_i\alpha_{k,i}\prod_{j}\tr(x_0u'_{i,j}) v'_i
\end{equation}
where $u'_{i,j}, v'_i$ are words in $x_1,\dots,x_n$.

By the special structure \eqref{e:expand1},
$$
T_{\ell k}(\lambda;\uX)^{\oplus \ell}
=T_{\ell k}\left( \tfrac1\ell I_\ell\otimes \lambda ,I_\ell\otimes \uX\right)
=f_{\ell k}\left( \tfrac1\ell I_\ell\otimes \lambda ,I_\ell\otimes \uX\right)
= f_k(\lambda,\uX)^{\oplus \ell} 
=T_k(\lambda,\uX)^{\oplus \ell}
$$
for all $\ell\in\N$, $k\ge d+1$ and $(\lambda,\uX)\in \cS(\R^k)\times\sbb{k}^n$.
By uniqueness of the $T_k$ we thus have $T_{\ell k}=T_k$ for all $\ell\in\N$, $k\ge d+1$, and consequently $T_k=T_{d+1}$ for all $k\ge d+1$.
The property $(b)$ ensures that 
the evaluations of $T_{d+1}$ and $T_k$ for $k\le d$ agree on $\sbb{k}^{n+1}$.
Thus
$$f=\sum_i\alpha_{d+1,i}\prod_{j}\sig(u'_{i,j}) v'_i \in\fatS$$
is the desired \ncstate polynomial.
\end{proof}

\begin{corollary}
\label{c:freean}
A sequence $(f_k)_{k\in\N}$ of polynomial maps $f_k:\cS(\R^k)\times \sbb{k}^n\to\R$ satisfies
\begin{enumerate}[\rm (a)]
    \item $f_k(O (\lambda;\uX)O^*)
    =f_k(\lambda;\uX)$ 
    for all $k\in\N$, $O\in \O_k$ and $(\lambda;\uX)\in \cS(\R^k)\times \sbb{k}^n$,
    \item $f_{\ell k}(\rho\otimes\lambda;\uX^{\oplus \ell})=f_k(\lambda;\uX)$ 
    for all $k,\ell\in\N$, $\rho\in \cS(\R^\ell)$ and $(\lambda;\uX)\in \cS(\R^k)\times \sbb{k}^n$,
    \item $\sup_k\deg f_k<\infty$,
\end{enumerate}
if and only if it is given by a \state polynomial.
\end{corollary}

\begin{proof}
Replacing $f_k$ with $\hat f_k:=f_k I_k$, Proposition \ref{p:freean}
implies $\hat f_k$ arises from an  \ncstate polynomial, say $\hat f$. 
Add a new variable $x_{n+1}$ and form the commutator
$g:=[\hat f,x_{n+1}]$. By assumption the image of $\hat f$
consists only of scalar matrices, so $g$ always evaluates to $0$ on all tuples $(\lambda;\uX)\in \cS(\R^k)\times \sbb{k}^{n+1}$.
Thus by Proposition \ref{prop:id}, $g=0$. Hence $\hat f$ cannot 
have any ``free'' nc variables not bound by $\sig$, i.e., 
$\hat f\in\skinnyS$.
\end{proof}

\section{Hilbert's 17th problem for state polynomials}
\label{sec:H17}

In this section we present a \state analog of Artin's solution to the celebrated Hilbert's 17th problem.
Let $\Omega\subset\skinnyS$ be the preordering \cite[Section 2.1]{marshallbook} generated by $\{\sig(hh^\star)\colon h\in \fatS\}$. 
That is, $\Omega$ is the smallest subset of $\skinnyS$ closed
under sums and products that contains all squares and $\{\sig(hh^\star)\colon h\in \fatS\}$.
Note that if $a\in\skinnyS$ then $a^2=aa^\star\sig(1)=\sig(aa^\star)$ since $a=a^\star$ and $\sig$ is a unital $\skinnyS$-linear map.
Hence $\Omega$ is the set of all sums of products of elements of the form $\sig(hh^\star)$ for $h\in\fatS$. 
Clearly, \state polynomials in $\Omega$ are nonnegative on $\cD_\emptyset$ and $\cD_\emptyset^\infty$. 
Theorem \ref{t:h17} below shows that every \state polynomial, nonnegative on $\cD_\emptyset$, is a quotient of elements in $\Omega$.
This is in stark contrast with trace polynomials: while the trace analog of Theorem \ref{t:h17} only holds for one matrix variable \cite[Corollary 3.8]{KPV}, it fails in general \cite[Proposition 6.4]{KSV}. We can pinpoint the crucial difference more precisely: the proof of Theorem \ref{t:h17} utilizes the Positivstellensatz in a free $\star$-algebra \cite{hkmConvex}; on the other hand, the tracial analog of this step would require a positive resolution of the renowned Connes' embedding conjecture, which has been recently refuted \cite{CECfalse}.

For $d\in\N$ let $\mx_d$ denote the set of words with length at most $d$, and $D=|\mx_d|=\frac{n^{d+1}-1}{n-1}$.
Let $H_d=(\sig(u v^\star))_{u,v\in\mx_d}\in\skinnyS^{D\times D}$.
The following is a well-known statement adapted to the notation of this paper.

\begin{lemma}\label{l:pd}
For each $d$ there exists $(\lambda;\uX)\in\vD_\emptyset$
such that $H_d(\lambda;\uX)$ is positive definite.
\end{lemma}

\begin{proof}
By \cite[Lemma 3.2]{hkmConvex}
there exists a unital $\star$-functional $L:\R\mx_{2d+2}\to\R$ such that $L(hh^\star)>0$ for all nonzero $h\in\R\mx_{d+1}$.
By \cite[Proposition 2.5]{hkmConvex}, there is $(\lambda;\uX)\in\vD_\emptyset$ such that $L(f)=\lambda(f(\uX))$ for all $f\in\R\mx_{2d}$. 
Then $H_d(\lambda;\uX)$ is positive definite by construction.
\end{proof}

\begin{proposition}\label{p:psdhank}
Every principal minor of $H_d$ is a quotient of two elements in $\Omega$.
\end{proposition}

\begin{proof}
Let $\Xi$ be the generic $D\times D$ symmetric matrix; that is, the entries of $\Xi$ are commuting indeterminates, related only by $\Xi$ being symmetric. 
Let $A$ be the real polynomial algebra generated by the entries of $\Xi$, and $T\subset A$ its real subalgebra generated by $\{\tr(\Xi^j)\colon 1\le j\le D \}$. Within $A^{D\times D}$ let $T[\Xi]$ denote the real subalgebra generated by the matrix $\Xi$ and multiples of identity $T\cdot I$.
Let $P$ be the preordering in $T$ generated by
$$\{
\tr(h(\Xi)^2),\,\tr(h(\Xi)\cdot\Xi\cdot h(\Xi))\colon h\in T[\Xi] \}.$$
Let $m\in T$ be an arbitrary principal minor of $\Xi$.
By \cite[Lemma 4.1 and Theorem 4.13]{klep2018positive} there exist $p,q\in P$ and $k\in\N$ such that 
\begin{equation}\label{e:pos}
qm=m^{2k}+p.
\end{equation}
Let $\w$ be the vector of words in $\mx_d$ (ordered degree-lexicographically); then $H_d$ is obtained by applying $\sig$ to $\w\w^\star \in\RX^{D\times D}$ entry-wise. If $h\in T[\Xi]\subset A^{D\times D}$, then $h(H_d)\in \skinnyS^{D\times D}$ and $h(H_d)\w\in \fatS^D$; moreover,
\begin{equation}\label{e:hanksquare}
\begin{split}
&\tr(h(H_d)^2) \in \Omega,\\
&\tr\left(h(H_d)\,H_d\,h(H_d)\right) =
\sum_{j=1}^D \sig\big(\left(h(H_d)\w\right)_j\left(h(H_d)\w\right)_j^\star\big) \in\Omega.
\end{split}
\end{equation}
If $m',p',q'$ are obtained from $m,p,q$ by replacing $\Xi$ with $H_d$, then $p',q'\in \Omega$ by \eqref{e:hanksquare}. Furthermore, $q'\neq0$ since the right-hand side of \eqref{e:pos} is strictly positive when evaluated at a positive definite \state evaluation of $H_d$, which exists by Lemma \ref{l:pd}. Therefore $m'$, a principal minor of $H_d$, is a quotient of elements in $\Omega$.
\end{proof}

The following is a solution to Hilbert’s 17th problem for \state polynomials.

\begin{theorem}\label{t:h17}
The following are equivalent for $a\in\skinnyS$:
\begin{enumerate}[\rm (i)]
\item $a(\lambda;\uX)\ge0$ for all $\uX\in\sbb{K}^n$ and vector states $\lambda\in \cS(\R^K)$, where $K=\frac{1}{n-1}(n^{\lceil\frac{1+\deg a}{2}\rceil}-1)$;
\item $a$ is nonnegative on $\cD_\emptyset^\infty$;
\item $a$ is a quotient of two elements in $\Omega$.
\end{enumerate}
\end{theorem}

\begin{proof}
(ii)$\Rightarrow$(i) Clear.

(i)$\Rightarrow$(iii) Suppose $a$ is not a quotient of elements in $\Omega$.
Let $d=\lceil\frac{1+\deg a}{2}\rceil$ and $R=\R[\sig(w)\colon w\in\mx_{2d}\setminus\{1\}]$.
Then $a\in R$, and $R$ is a finitely generated polynomial ring.
Let $M\subset R$ be the set of all principal minors of $H_d$. 
By Proposition \ref{p:psdhank}, $a$ is not a quotient of elements in the preordering in $R$ generated by $M$.
By the Krivine–Stengle Positivstellensatz 
\cite[Theorem 2.2.1]{marshallbook}
there exists a homomorphism $\varphi:R\to\R$ such that $\varphi(a)<0$ and $\varphi(M)\subset \R_{\ge0}$. Applying $\varphi$ entry-wise to $H_d$ therefore results in a positive semidefinite matrix.
Define $L:\R\mx_{2d}\to\R$ as $L(f)=\varphi(\sig(f))$. Then $L$ is a unital $\star$-functional, and $L(gg^\star)\ge0$ for $g\in\R\mx_d$. 
By Lemma \ref{l:pd} there exists a unital $\star$-functional $L':\R\mx_{2d}\to\R$ such that $L'(gg^\star)>0$ for nonzero $g\in\R\mx_d$. For every $\varepsilon\in(0,1)$ denote the unital $\star$-functional $L_\varepsilon = (1-\varepsilon)L+\varepsilon L'$; then $L_\varepsilon(gg^\star)>0$ for all nonzero $g\in\R\mx_d$.
By \cite[Proposition 2.5]{hkmConvex} there exists $(\lambda_\varepsilon;\uX_\varepsilon)\in\cS(\R^K)\times \sbb{K}^n$, with $K=\dim\R\mx_{d-1}=\frac{n^d-1}{n-1}$ and $\lambda_\varepsilon$ a vector state, such that $L_\varepsilon(f)=\lambda_\varepsilon (f(\uX_\varepsilon))$ for all $f\in\R\mx_{2d-1}$. 
Then 
$$\lim_{\varepsilon\to 0}a(\lambda_\varepsilon;\uX_\varepsilon)=
a(\lambda_0;\uX_0)=\varphi(a)<0$$
and so $a(\lambda_\varepsilon;\uX_\varepsilon)<0$ for some $\varepsilon\in (0,1)$.

(iii)$\Rightarrow$(ii) Let $a=p/q$ for some $p,q\in \Omega$ with $q\neq0$.
Clearly $p$ and $q$ are nonnegative on $\cD_\emptyset^\infty$.
Let $\uX$ be a tuple of self-adjoint bounded operators on an infinite-dimensional separable Hilbert space $\cH$, and let $\lambda$ be a state on $\cB(\cH)$.
Suppose $a(\lambda;\uX)<0$. Since $q\neq0$, by Proposition \ref{prop:id} there exist $\uY\in\sbb{k}^n$ and a $\mu\in\cS(\R^k)$ such that $q(\mu;\uY)\neq0$. 
Let $\iota:\R^k\to\cH$ be an isometry, and define $X_j'=\iota\circ Y_j \circ \iota^*$ and $\lambda'=\mu \circ \iota^*$.
Then $q(\lambda';\uX')\neq0$.
For $\varepsilon\in (0,1)$ set $\lambda_\varepsilon=(1-\varepsilon)\lambda+\varepsilon \lambda'$ and
$\uX_\varepsilon=(1-\varepsilon)\uX+\varepsilon \uX'$.
Since $q(\lambda_1;\uX_1)\neq0$ and $q$ is a polynomial in state symbols, we have $q(\lambda_\varepsilon;\uX_\varepsilon)\neq0$ for all but finitely many $\varepsilon\in (0,1)$. On the other hand,
$\lim_{\varepsilon\to0}a(\lambda_\varepsilon;\uX_\varepsilon)
=a(\lambda_0;\uX_0)<0$. Therefore there exists $\varepsilon\in (0,1)$ such that $q(\lambda_\varepsilon;\uX_\varepsilon)\neq0$ and $a(\lambda_\varepsilon;\uX_\varepsilon)<0$, so
$$0>a(\lambda_\varepsilon;\uX_\varepsilon)
=\frac{p(\lambda_\varepsilon;\uX_\varepsilon)}{q(\lambda_\varepsilon;\uX_\varepsilon)} \ge0,$$
a contradiction.
\end{proof}

Theorem \ref{t:h17} shows that {\it every} state polynomial inequality valid on all matrices is an algebraic consequence of states on hermitian squares, i.e., it is built from expressions $\sig(hh^\star)$ using addition, multiplication and inversion. 
In particular, a state polynomial inequality valid on all matrices is automatically valid on all operators.
A well-known example is the Cauchy-Schwarz inequality, which admits an algebraic certificate
\begin{equation}\label{e:cs_certificate}
\sig(x_1^2)\sig(x_2^2)-\sig(x_1x_2)^2=
\frac{\sig\Big(
\big(\sig(x_1^2)x_2-\sig(x_1x_2)x_1\big)^2
\Big)}{\sig(x_1^2)}.
\end{equation}
Alternatively, one can recognize $\sig(x_1^2)\sig(x_2^2)-\sig(x_1x_2)^2$ as the $2\times 2$ principal minor of $H_1$ indexed by $x_1,x_2$.

\begin{example}\label{e:complicated}
A less evident example of a globally nonnegative \state polynomial is
\begin{align*}a=\ &\big(
\sig(x_1^2)\sig(x_2^2)-\sig(x_1x_2)^2
\big)\sig(x_2x_1^2x_2)
+2\sig(x_1x_2)\sig(x_2x_1x_2)\sig(x_1^2x_2) \\
&-\sig(x_1^2)\sig(x_2x_1x_2)^2
-\sig(x_2^2)\sig(x_1^2x_2)^2.
\end{align*}
Let us give two arguments for nonnegativity of $a$.
Firstly, consider a principal submatrix of $H_2$,
$$s= \begin{pmatrix}
\sig(x_1^2) & \sig(x_1x_2) & \sig(x_1^2x_2) \\
\sig(x_1x_2) & \sig(x_2^2) & \sig(x_2x_1x_2) \\
\sig(x_1^2x_2) & \sig(x_2x_1x_2) & \sig(x_2x_1^2x_2)
\end{pmatrix}\in\skinnyS^{3\times 3}.$$
A calculation shows that $a=\det(s)$. Since $H_2(\lambda,\uX)$ is positive semidefinite for every $(\lambda;\uX)\in\vD_\emptyset$, it follows that $a$ is nonnegative. Secondly, we sketch how to directly obtain a certificate for nonnegativity of $a$ in terms of Theorem \ref{t:h17}. Let
$$\sigma_2 = 
\Big(\sig(x_1^2)\sig(x_2x_1^2x_2)-\sig(x_1^2x_2)^2\Big)
+\Big(\sig(x_1^2)\sig(x_2^2)-\sig(x_1x_2)^2\Big)
+\Big(\sig(x_2^2)\sig(x_2x_1^2x_2)-\sig(x_2x_1x_2)^2\Big)$$
which is one of the coefficients of the characteristic polynomial of $s$. Note that $\sigma_2$ is a sum of three terms obtained from \eqref{e:cs_certificate} by substitution, and therefore in particular a sum of three quotients of elements in $\Omega$.
If $h_1,h_2,h_3\in\fatS$ are given by
$$
\begin{pmatrix}
h_1 \\ h_2 \\ h_3
\end{pmatrix}=
(s^2-\tr(s)s+\sigma_2 I_3)\cdot\begin{pmatrix}
x_1 \\ x_2 \\ x_2x_1
\end{pmatrix} \in\fatS^3,
$$
then a direct (but tedious) calculation shows that
$$a=\frac{\sig(h_1h_1^\star)+\sig(h_2h_2^\star)+\sig(h_3h_3^\star)}{\sigma_2},
$$
so $a$ is a quotient of elements in $\Omega$, and therefore nonnegative.
The choice of $h_j$ is inspired by the proof of Proposition \ref{p:psdhank} and \cite[Example 6.1]{klep2018positive}.
\end{example}

\section{Archimedean Positivstellensatz for state polynomials}
\label{sec:noncyclic}

In this section we give a version of Putinar's Positivstellensatz \cite{Putinar1993positive} for \state polynomials subject to archimedean constraints, Theorem \ref{thm:nocyc}, 
which is later applied to \state polynomial optimization in Section \ref{sec:hierarchy}.
First we address which functionals $\RX\to\R$ are given by states and evaluations on tuples of bounded operators.
The following is a variant of the well-known Gelfand-Naimark-Segal (GNS) construction \cite[Section I.9]{Tak02}.

\begin{proposition}
\label{prop:moment}
Let $L: \RX\to \R$ be a unital $\star$-functional. If
\begin{enumerate}[\rm (a)]
	\item $L(pp^\star)\ge0$ for all $p\in \RX$, and
	\item there is $N>0$ such that $L(p(N-x_1^2-\cdots-x_n^2)p^\star)\ge0$ for all $p\in \RX$,
\end{enumerate}
then there exist a vector state $\lambda\in\cS(\cH)$ and a tuple of self-adjoint operators $\uX\in\cB(\cH)^n$ such that $L(p)=\lambda(p(\uX))$ for all $p\in \RX$.
\end{proposition}

\begin{proof}
Apply \cite[Theorem 1.27]{burgdorf16} to the quadratic module in $\RX$ generated by $\{N-x_1^2-\cdots-x_n^2\}$.
\end{proof}

\begin{remark}\label{r:nontr}
It is easy to see that (b) in Proposition \ref{prop:moment} 
can be replaced by 
\begin{enumerate}
\item[(b')] there is $N>0$ such that $L(ww^\star)\le N^{|w|}$ for all $w\in\mx$.
\end{enumerate}
On the other hand, it cannot be replaced by
\begin{enumerate}
\item[(b'')] there is $N > 0$ such that $L(x_j^{2k})\le N^{k}$ for all $j=1,\dots,n$ and $k\in\N$,
\end{enumerate}
as in the tracial setup \cite[Proposition 3.2]{KMV}. Namely, for non-tracial states not all mixed moments can be bounded with univariate moments. For example,
let
$$v=\left(\begin{smallmatrix}
1 \\ 0 \\ 0\\ 0
\end{smallmatrix}\right),
\quad
X_1=\left(\begin{smallmatrix}
0&1&0&0 \\
1&0&0&0 \\
0&0&\gamma&0 \\
0&0&0&\gamma \\
\end{smallmatrix}\right),
\quad
X_2=\left(\begin{smallmatrix}
	1&0&0&0 \\
	0&0&0&\gamma \\
	0&0&1&0 \\
	0&\gamma&0&0 \\
\end{smallmatrix}\right),
\quad
X_3=\left(\begin{smallmatrix}
	0&0&0&1 \\
	0&\gamma&0&0 \\
	0&0&\gamma&0 \\
	1&0&0&0 \\
\end{smallmatrix}\right).
$$
Then $\langle X_j^{2k}v,v\rangle=1$ for all $j,k$ but $\langle X_1X_2X_3v,v\rangle =\gamma$ can be arbitrarily large.
\end{remark}

Due to Proposition \ref{prop:moment}, we focus on \state polynomial positivity subject to balanced constraint sets with the following property. 
We say that
$C\subseteq \fatS$ is {\it algebraically bounded} if there is
$N>0$ such that
$$N-x_1^2-\cdots-x_n^2 = \sum_ip_ic_ip_i^\star$$
for some $c_i\in \{1\}\cup C\cap\RX$ and $p_i\in\RX$
(in other words, $C\cap\RX$ generates an archimedean quadratic module in $\RX$).

Next we turn to a notion from real algebra \cite{marshallbook,sche}.
A subset $\cM \subseteq \skinnyS$ is called a quadratic module if $1 \in \cM$, $\cM + \cM \subseteq \cM$ and $a^2 \cM$ for all $a \in \skinnyS$.
For $M\subset \skinnyS$ let $\QM(M)$ denote the quadratic module generated by $M$.
Given a quadratic module $\cM \subseteq \skinnyS$ that is archimedean (i.e., for each $f\in\skinnyS$ there is $m>0$ such that $m\pm f\in\cM$), we consider the real points of the real spectrum $\operatorname{Sper}_{\cM}\skinnyS$, namely the set $\chi_{\cM}$ defined by\looseness=-1
\begin{align}
\label{eq:chiM}
\chi_{\cM} := \{\varphi : \skinnyS \to \R \colon \varphi \text{ homomorphism,}  \ \varphi(\cM) \subseteq \R_{\geq 0}, \ \varphi(1) = 1 \}.
\end{align}

The next proposition is the well-known Kadison-Dubois representation theorem, see, e.g., \cite[Theorem 5.4.4]{marshallbook}.

\begin{proposition}
\label{prop:KD}
Let $\cM \subseteq \skinnyS$ be an archimedean quadratic module.
Then, for all $a \in \skinnyS$, one has
$$\forall \varphi \in \chi_{\cM}  \quad \varphi(a) \geq 0 \qquad \Leftrightarrow \qquad \forall \varepsilon > 0 \quad a + \varepsilon \in \cM.$$
\end{proposition}

Since the algebra $\skinnyS$ is not finitely generated, it is in general not straightforward to determine if a quadratic module in $\skinnyS$ is archimedean.
Nevertheless, the next lemma shows that quadratic modules arising from algebraically bounded sets are archimedean.
To $C\subseteq\fatS$ we assign
$$C^\sig:=\{\sig (pcp^\star)\colon p\in\RX, c\in \{1\}\cup C \}
\subseteq\skinnyS.$$

\begin{lemma}\label{l:arch}
If $C$ is balanced and algebraically bounded then the quadratic module $\QM(C^\sig)\subseteq \skinnyS$ is archimedean.
\end{lemma}

\begin{proof}
It suffices to show that the generators of $\skinnyS$ are bounded with respect to $\QM(C^\sig)$ \cite{cimp}, i.e., that for every $w\in\mx$ there exists $m>0$ such that $m\pm\sig(w)\in \QM(C^\sig)$.
Since
$$N\sig(ww^\star)-\sig(wx_j^2w^\star)
=\sig\left(w(N-x_j^2)w^\star \right)\in \QM(C^\sig),
$$
induction on $|w|$ implies that for every $w$ there is $m'>0$ such that $m'-\sig(ww^\star)\in \QM(C^\sig)$.
Then
\begin{equation*}
\tfrac14+m'\pm\sig(w) = 
\sig\left((\tfrac12\pm w)(\tfrac12\pm w)^\star\right)
+m'-\sig(ww^\star) \in \QM(C^\sig).\qedhere
\end{equation*}
\end{proof}

We are now ready to prove an analog of the noncommutative Helton-McCullough Positivstellensatz \cite{Helton04} for \state polynomials subject to \ncstate constraints.

\begin{theorem}\label{thm:nocyc}
Let $C\subseteq \fatS$ be balanced and algebraically bounded.
Then for $a\in\skinnyS$ the following are equivalent:
\begin{enumerate}[\rm (i)]
\item\label{it:i3} $a(\lambda;\uX)\geq 0$ for all 
$(\lambda;\uX) \in \vD_C^\infty$;\vspace{0.25ex}
\item\label{it:i2} $a(\lambda;\uX)\geq 0$ for all 
$(\lambda;\uX) \in \cD_C^\infty$;\vspace{0.25ex}
\item\label{it:i1} $a+\varepsilon \in \QM(C^\sig)$ for all $\varepsilon>0$.
\end{enumerate}
\end{theorem}

\begin{proof}
\ref{it:i1}$\Rightarrow$\ref{it:i2} If $(\lambda;\uX) \in \mathcal{D}_C^\infty$, then
$$c(\lambda;\uX)\succeq0,\qquad \lambda \big(p(\uX)p(\uX)^*\big)\succeq0$$
for all $c\in C$ and $p\in\RX$, hence $s(\lambda;\uX)\ge0$ for all $s\in\QM(C^\sig)$,
and so $a(\lambda;\uX)\ge0$.

\ref{it:i2}$\Rightarrow$\ref{it:i3} Clear.

\ref{it:i3}$\Rightarrow$\ref{it:i1} Suppose $a+\varepsilon \notin \QM(C^\sig)$ for some $\varepsilon>0$. By Proposition \ref{prop:KD} there exists a unital homomorphism $\varphi: \skinnyS\to\R$ with $\varphi(\QM(C^\sig))\subseteq \R_{\ge0}$ and $\varphi(a)<0$. Hence
$$\varphi(\sig (pp^\star))\ge0,\qquad \varphi(\sig (p(N-x_1^2-\cdots-x_n^2)p^\star))\ge0$$
for all $p\in\RX$. 
Consider the unital $\star$-functional $L:\RX\to\R$ given by $L(p)=\varphi(\sig(p))$.
By Proposition \ref{prop:moment}, there exist a vector state $\lambda\in\cS(\cH)$ and $\uX=\uX^*\in\cB(\cH)^n$ such that 
$L(p)=\lambda(p(\uX))$ for all $p\in\RX$.
Therefore $\varphi(b)=b(\lambda;\uX)$ for all $b\in \skinnyS$.
Let $v\in\cH$ be a unit vector such that $\lambda(Y)=\langle Yv,v\rangle$ for all $Y\in\cB(\cH)$, and 
$P\in\cB(\cH)$ be the orthogonal projection onto $\overline{\{p(\uX)v\colon p\in\RX \}}$. 
Then $\langle q(\uX)v,v\rangle=\langle q(P\uX P)v,v\rangle$ for all $q \in \RX$.
Thus we can without loss of generality assume $PX_j=X_j$; that is, the operators $X_j$ can be replaced by their compressions $PX_jP$.
Hence
$$\left\langle c(\lambda;\uX)p(\uX)v,p(\uX)v\right\rangle
=\varphi(\sig(p^\star cp))\ge0 
\qquad \text{for } c\in C,\, p\in\RX$$
together with $(I-P)X_j=0$ implies $c(\lambda;\uX)\succeq0$, and so $(\lambda,\uX) \in \vD_C^\infty$. Finally, $a(\lambda;\uX)=\varphi(a)<0$.
\end{proof}

Given $a \in \skinnyS$ and $c,p \in \fatS$ we have $a^2 \sig(pcp^\star) = \sig ((a p)c(a p)^\star)$. Therefore
\begin{equation}\label{eq:niceMC}
\QM(C^\sig) \subseteq \left\{
\sum_{i=1}^K \sig(f_ic_if_i^\star)\colon K\in\N,\ f_i\in\fatS,\  c_i\in \{1\}\cup C\right\},
\end{equation}
an inclusion we shall make use of in Sections \ref{ssec:stateMod} and \ref{sec:constrained} below.

\subsection{Quadratic \state modules}\label{ssec:stateMod}

In this section we provide an alternative form of the Positivstellensatz for \state polynomials that underlines their noncommutative roots.
A subset $\cM \subseteq \fatS$ is called a {\em quadratic \state module} if
$$1 \in \cM,\ \cM+\cM\subseteq \cM,\  
f \cM f^\star \subseteq \cM\ \forall f\in\fatS,\ \sig(\cM)\subseteq\cM.$$
Given $C\subseteq\fatS$ let $\cQ(C)$ be the quadratic \state module generated by $C$, i.e., the smallest quadratic \state module in $\fatS$ containing $C$.
We start with an alternative description of quadratic \state modules.

\begin{lemma}
\label{lemma:cycgen}
Let $C\subseteq\fatS$.
\begin{enumerate}[\rm (1)]
	\item Elements of $ \cQ(\emptyset)$ are precisely sums of
	$$\sig(h_1h_1^\star)\cdots \sig(h_\ell h_\ell^\star)h_0h_0^\star$$
	for $h_i\in\fatS$.
	\item Elements of $\cQ(C)$ are precisely sums of
	$$q_1,\quad h_1c_1h_1^\star,\quad \sig(h_2c_2h_2^\star)q_2$$
	for $h_i\in \fatS$, $q_i\in \cQ(\emptyset)$, $c_i\in C$.
	\item Elements of $\sig(\cQ(C))=\cQ(C)\cap\skinnyS$ are precisely sums of
	$$\sig(h_1h_1^\star)\cdots \sig(h_\ell h_\ell^\star)\sig(h_0 c h_0^\star )$$
	for $h_i\in\fatS$ and $c\in C\cup\{1\}$.
	\\ In particular, $\sig(\cQ(\emptyset))$ is the preordering $\Omega$ from Section \ref{sec:H17}. 
\end{enumerate}
\end{lemma}

\begin{proof}
Straightforward.
\end{proof}

Note that while $\sig(\cQ(\emptyset))$ is a preordering in $\skinnyS$, this is not necessarily the case for $\sig(\cQ(C))$ in general (namely, $\sig(\cQ(C))$ does not need to contain elements of the form $\sig(h_1 c_1 h_1^\star )\sig(h_2 c_2 h_2^\star )$ for $c_i\in C$).
A quadratic \state module $\cM$ is called {\em archimedean} if for every $f \in \SymS$ there exists $N>0$ such that $N \pm f \in \cM$ (note that even though $\cM$ might not be contained in $\SymS$, we only consider symmetric $f$).

\begin{proposition}
\label{prop:cyclicqmarch}
A quadratic \state module $\cM$ is archimedean if and only if there exists $N>0$ such that $N - x_1^2-\cdots-x_n^2 \in \cM$.
\end{proposition}
\begin{proof}
The forward implication is obvious. For the converse, note that the set $\cM \cap \RX$ is an archimedean quadratic module. Thus, for all $p=p^\star\in \RX$ there exists  $M>0$ such that 
\begin{align}
\label{eq:cyclicqmarch1}
M \pm p \in \cM \cap \RX \,.
\end{align}
In addition, the set $H$ of bounded elements, defined by
\[
H = \{h \in \fatS \colon \exists M>0 \text{ s.t. } M -hh^\star \in \cM \} \,,
\]
is closed under involution, addition, subtraction and multiplication, i.e., is a $\star$-subalgebra of $\fatS$ \cite{cimp}.
A symmetric element $f\in\fatS$ is in $H$ if and only if there is some $M>0$
with $M\pm f\in \cM$.

For every $w\in\mx$ we have
\begin{equation}\label{eq:cs1}
\sig(ww^\star)-\sig(w)^2 = \sig\big((w-\sig(w))(w-\sig(w))^\star\big)\in \cM.
\end{equation}
By \eqref{eq:cyclicqmarch1} and the fact that $\cM$ is preserved under $\sig$, 
there is some $M>0$ with $M-\sig(ww^\star)\in \cM$.
Adding this to \eqref{eq:cs1} yields
$M-\sig(w)^2\in \cM$. Thus, by the definition of $H$, $\sig(w)\in \cM$.
The desired result now follows since $H$ is a subalgebra of $\fatS$.
\end{proof}

The following is the \state version of Theorem \ref{thm:nocyc}. Note that while the constraints in Corollary \ref{c:psatz} are \ncstate polynomials, the objective function needs to be a \state polynomial (cf.~\cite[Example 4.6]{KMV}).

\begin{corollary}
\label{c:psatz}
Let $\cM \subseteq \fatS$ be an archimedean quadratic \state module and $a \in \skinnyS$. The following are equivalent:
\begin{enumerate}[\rm (i)]
\item\label{it:j3} $a(\lambda;\uX)\geq 0$ for all $(\lambda,\uX) \in \vD_{\cM}^\infty$;\vspace{0.25ex}
\item\label{it:j2} $a(\lambda;\uX)\geq 0$ for all $(\lambda,\uX) \in \mathcal{D}_{\cM}^\infty$;\vspace{0.25ex}
\item\label{it:j1} $a+\varepsilon \in \cM$ for all $\varepsilon>0$.
	\end{enumerate}
\end{corollary}

\begin{proof}
By Proposition \ref{prop:cyclicqmarch}, there exists $N>0$ such that $N-x_1^2-\cdots-x_n^2\in\cM$.
Furthermore, $\QM(\cM^\sig)\subseteq\cM$ by \eqref{eq:niceMC}.
Therefore \ref{it:j3}$\Rightarrow$\ref{it:j1} follows from Theorem \ref{thm:nocyc}, and \ref{it:j1}$\Rightarrow$\ref{it:j2}$\Rightarrow$\ref{it:j3} 
is clear.
\end{proof}

\subsection{Nonexistence of a Krivine-Stengle or Schm\"udgen Positivstellensatz for \state polynomials}\label{ssec:nogo}

We finish this section by commenting on two classical non-archimedean Positivstellens\"atze from real algebraic geometry, and how their straightforward (albeit possibly na\"ive) analogs for \state polynomials fail.

A quadratic \state module $\cP\subseteq\fatS$ is a \emph{\state preordering} if $\sig(\cP)\cdot\cP\subseteq\cP$.
Note that $\sig(\cP)=\cP\cap\skinnyS$ is then a preordering (in the usual sense) in $\skinnyS$.
Moreover, if $\cP$ is a \state preordering generated by $C$, then $\sig(\cP)$ is the preordering generated by $\sig(hch^\star)$ for $c\in\{1\}\cup C$.

Consider the following two statements about a finitely generated \state preordering $\cP$ and $a\in\skinnyS$, which would be analogs of the Krivine-Stengle Positivstellensatz \cite[Theorem 2.2.1]{marshallbook} and the Schm\"udgen Positivstellensatz \cite[Corollary 6.1.2]{marshallbook}, respectively:
\begin{enumerate}[(A)]
\item if $a|_{\cD_\cP^\infty}\ge0$ then there exist $p_1,p_2\in\cP$ and $k\in\N$ such that $ap_1=a^{2k}+p_2$;
\item if $\cD_\cP^\infty$ is bounded in operator norm and $a|_{\cD_\cP^\infty}\ge\varepsilon$ for some $\varepsilon>0$ then $a\in\cP$.
\end{enumerate}

\begin{example}
Let $\cP$ be the \state preordering generated by
$$\Big\{\pm\big(1+[x_1,x_2]^2\big),
\pm\big[[x_1,x_2],x_1\big],
\pm\big[[x_1,x_2],x_2\big] \Big\},$$
and $a=-\sig(x_1)$.
Then
\begin{enumerate}[\rm (1)]
    \item $\cD_\cP^\infty=\emptyset$;
    \item there is a homomorphism $\varphi:\skinnyS\to\R$ such that $\varphi(\sig(\cP))= \R_{\ge0}$ and $\varphi(a)<0$;
    \item The above implications (A) and (B) both fail.
\end{enumerate}
\end{example}

\begin{proof}
(1) Let $\cH$ be a complex Hilbert space, and suppose there exist $X_1,X_2\in\cB(\cH)$ such that
$[X_1,X_2]$ commutes with $X_1,X_2$ and
\begin{equation}\label{e:preWeyl}
I+[X_1,X_2]^2=0.
\end{equation}
By the GNS construction (applied with any state on $\cB(\cH)$) one can then without loss of generality assume that there exists $v\in\cH$ such that $\{p(\uX)v\colon p\in\RX\}$ is dense in $\cH$. Then $[X_1,X_2]$ is central in $\cB(\cH)$, and so \eqref{e:preWeyl} implies $[X_1,X_2]=\pm iI$. But this contradicts nonexistence of bounded representations of the Weyl algebra \cite[Example VIII.5.2]{reed80}. Therefore $\cD_\cP^\infty=\emptyset$.

(2) The standard (Bargmann-Fock) unbounded $\star$-representation of the Weyl algebra is given by the unbounded operator $T$ acting on the complex Hilbert space $\ell^2$ with the orthonormal basis $\{e_n\}_{n\in\N}$ as $T e_n=\sqrt{n}e_{n+1}$. The domain of this representation contains $\bigoplus_{n\in\N}\C e_n$, and $T$ satisfies $T^*T-TT^*=I$.
Let $v=\frac{1}{\sqrt{2}}(e_1+e_2)$ and $X_1=\frac12 (T+T^*)$, $X_2=\frac{1}{2i} (T-T^*)$. Then $X_1,X_2$ are self-adjoint unbounded operators, and $[X_1,X_2]=iI$. Define a homomorphism $\varphi:\fatS\to\R$ by
$$\varphi (\sig(w)) = \operatorname{Re}\,\langle w(X_1,X_2)v,v\rangle$$
for $w\in\mx$. By the construction, $\varphi(\sig(\cP))=\R_{\ge0}$ and $\varphi(a)=-\frac12$.

(3) Note that $\cD_\cP^\infty$ is bounded and $a|_{\cD_\cP^\infty}\ge 1$ vacuously by (1). Then (B) fails since $a\notin\sig(\cP)=\cP\cap\skinnyS$ by (2).
If $ap_1=a^{2k}+p_2$ for $p_1,p_2\in\cP$, then 
$$ 0\ge\varphi(a)\varphi(\sig(p_1))=\varphi(a)^{2k}+\varphi(\sig(p_2))>0,$$
a contradiction; thus  also (A) fails.
\end{proof}

\section{Optimization of \state polynomials}
\label{sec:hierarchy}

In this section we apply Theorem \ref{thm:nocyc} to optimization of \state objective functions subject to \ncstate constraints. 
Doing so, we obtain a converging hierarchy of SDP relaxations in Section \ref{sec:constrained}, and its dual in Section \ref{sec:constrained_dual}. 
When flatness occurs in the latter hierarchy, one can extract a finite-dimensional minimizer as shown in Section \ref{sec:gns}. 
Finally, Section \ref{ssec:sparse} presents sparsity-based
approaches to reducing the sizes of the constructed SDP hierarchies.

Recall that the degree of a $\fatS$-word $\prod_j\sig(u_j) v$ equals $|v|+\sum_j|u_j|$, and the degree of $f \in \fatS$ is the maximal degree of $\fatS$-words in the expansion of $f$.
Let $\fatS\!_d$ be subspace of \ncstate polynomials of degree at most $d$, and denote $\bsigma(n,d)=\dim \fatS\!_d$. Note that the number of $\fatS$-words of degree $d\in\N$ is bounded below by $n^d$ (the number of words from $\mx$ of length $d$) and above by $\frac12 (2n)^d$ (the number of ordered lists of words from $\mx\setminus\{1\}$ whose lengths add to $d$). This gives a coarse estimate $n^d\le \bsigma(n,d)\le (2n)^{d+1}$. Also, let $\skinnyS_d=\fatS\!_d\cap\skinnyS$.
By $\W^{\fatS}_d$ we denote the vector of all $\fatS$-words of degree at most $d$ with respect to the degree lexicographic order; note that $\W^{\fatS}_d$ is of length $\bsigma (n,d)$.
Given $c \in \fatS$ denote $d_c := \lceil \frac{\deg c}{2} \rceil$.

Throughout the rest of the paper we restrict ourselves to constraint sets $C\subset \fatS$ such that $C\cap \fatS_d$ is finite for all $d\in\N$. 
In polynomial optimization one typically focuses on finite sets of constraints, but this slightly more general setup is needed later in Section \ref{sec:networks}.
\subsection{SDP hierarchy for \state polynomial optimization}
\label{sec:constrained}

For a balanced $C\subseteq\fatS$ and $d\in\N$ define
\begin{equation}\label{eq:truncated}
\cM(C)_d:=\left\{
\sum_{i=1}^{K} \sig(f_ic_if_i^\star)\colon K\in\N,\ f_i\in\fatS,\ c_i \in \{1\}\cup C,\ \deg(f_ic_if_i^\star)\le 2d
\right\}.
\end{equation}
By \eqref{eq:truncated} it is clear that checking membership in $\cM(C)_d$ can be formulated as an SDP. Indeed, let $C\cap\fatS_d=\{c_1,\dots,c_\ell\}$, and $a\in\skinnyS_d$. Then 
$a\in\cM(C)_d$ if and only if there exist positive semidefinite matrices $A_0,\dots,A_\ell$, where $A_0$ is of size $\Delta(n,d)$ and $A_i$ is of size $\Delta(n,d-d_{c_i})$ for $i=1,\dots,\ell$, such that
$$a=\sig\left({\W^{\fatS}_d}^\star
\cdot A_0\cdot \W^{\fatS}_d\right)
+\sum_{i=1}^\ell
\sig\left({\W^{\fatS}_{d-d_{c_i}}}^\star
\cdot(c_iA_i)\cdot
\W^{\fatS}_{d-d_{c_i}}\right).
$$
Furthermore, $\cM(C)_d$ for $d=1,2,\dots$ is an increasing sequence of convex cones whose union by \eqref{eq:niceMC} contains the quadratic module $\QM(C^\sig)$ from Section \ref{sec:noncyclic}.

Given a \state polynomial $a\in\skinnyS$, one can use $\cM(C)_d$ to design a hierarchy of SDP relaxations for minimizing $a$ over the \state semialgebraic set $\mathcal{D}_{C}^\infty$.
Let us define  $a_{\min}$ and $a_{\min}^{\infty}$ as follows:
\begin{align}
\label{eq:pure_constr} a_{\min} &:= \inf \{ a(\lambda;\uX) \colon (\lambda,\uX) \in \mathcal{D}_{C} \} \,, \\ 
\label{eq:pure_constr1} a_{\min}^{\infty} &:=
\inf \{ a(\lambda;\uX) \colon (\lambda,\uX) \in \mathcal{D}_{C}^\infty \}\,. 
\end{align}
Since $\mathcal{D}_{C}$ embeds into $\mathcal{D}_{C}^\infty$, one has $a_{\min}^\infty \leq a_{\min}$.
One can approximate $a_{\min}^\infty$ from below via the following hierarchy of SDPs, indexed by $d \geq \frac{\deg a}{2}$: 
\begin{align}
\label{eq:pure_constr_dual}
a_{\min,d} = \sup \{ m \colon a - m \in \cM(C)_d \} \,.
\end{align}

\begin{corollary}
\label{cor:pure_cvg}
If $C$ is balanced and algebraically bounded, the SDP hierarchy \eqref{eq:pure_constr_dual} provides a sequence of lower bounds $(a_{\min,d})_{d \geq \frac{\deg a}{2}}$  monotonically converging to $a_{\min}^\infty$.
\end{corollary}
\begin{proof}
By \eqref{eq:truncated} and \eqref{eq:pure_constr_dual}, it is clear that $a_{\min,d}\le a_{\min}^\infty$.
As $\cM(C)_d \subseteq \cM(C)_{d+1}$, one has $a_{\min,d} \leq a_{\min,d+1}$.
Furthermore, Theorem~\ref{thm:nocyc} implies that for each each $m \in \N$, there exists $d(m) \in \N$ such that $a - a_{\min}^\infty + \frac{1}{m} \in \cM(C)_{d(m)}$.
Thus one has
\[
a_{\min}^\infty - \frac{1}{m} \leq  a_{\min,d(m)} \,,
\]
which implies that
\[
\lim_{d \to \infty} a_{\min,d}  = a_{\min}^\infty
 \,.\qedhere\]
\end{proof}

\subsection{Duality and \state Hankel matrices}\label{sec:constrained_dual}

Next, we introduce \state Hankel and localizing matrices, which can be viewed as analogs of the noncommutative Hankel and localizing matrices (see e.g. \cite[Lemma~1.44]{burgdorf16}). 
To $c\in \fatS$ and a linear functional $L:\skinnyS_{2 d} \to \R$, we associate the following two matrices:
\begin{enumerate}[(a)]
\item the {\em Hankel matrix} $\M_d(L)$ is the symmetric matrix of size $\bsigma (n,d)$, indexed by $\fatS$-words $u,v \in \fatS\!_{d}$, with 
$(\M_d(L) )_{u,v}=L (\sig(u^\star v))$;
\item the {\em localizing matrix} $\M_{d -  d_c}(c \, L)$ is the symmetric matrix of size $\bsigma (n,d - d_c)$, indexed by $\fatS$-words  $u,v \in \fatS\!_{d - d_c}$, with $(\M_{d - d_c}(c \, L))_{u,v}= L ( \sig(u^\star c v))$.
\end{enumerate}

Note that the localizing matrix associated to $L$ and $1$ is simply the Hankel matrix associated to $L$.
\begin{definition}
\label{def:hankel_condition}
A matrix $M$ indexed by $\fatS$-words of degree  $\leq d$ satisfies the \emph{\state Hankel condition} if and only if 
\begin{align}
\label{eq:hankel_condition}
M_{u,v} = M_{w,z} \text{ whenever } \sig(u^\star v) = \sig (w^\star z) \,.
\end{align}
\end{definition}
\begin{remark}
\label{rk:Hankelbij}
Linear functionals on $\skinnyS_{2 d}$ and matrices from $\sbb{\bsigma(n,d)}$ satisfying the \state condition \eqref{eq:hankel_condition} are in bijective correspondence.
To a linear functional $L : \skinnyS_{2 d} \to \R$, one can assign the matrix $\M_d(L)$, defined by $(\M_d(L) )_{u,v} = L (\sig(u^\star v))$, satisfying the \state Hankel condition, and vice versa. 
\end{remark}
The positivity of $L$ relates to the positive semidefiniteness of its Hankel matrix $\M_d(L)$. The proof of the following lemma is straightforward and analogous to its free counterpart \cite[Lemma 1.44]{burgdorf16}.
\begin{lemma}
\label{lemma:pureHankel}
Given a linear functional $L : \skinnyS_{2 d} \to \R$, one has $L(\sig(f^\star f)) \geq 0$ for all $f \in \fatS\!_{d}$, if and only if, $\M_d (L) \succeq 0$. 
Given $c \in \fatS$, one has $L(\sig(f^\star \, c \, f)) \geq 0$ for all $f \in \fatS\!_{d - d_c }$, if and only if, $\M_{d - d_c} (c \, L) \succeq 0$.
\end{lemma}

We are now ready to state the dual SDP of \eqref{eq:pure_constr_dual}.
\begin{lemma}
\label{lemma:pure_dual}
The dual of \eqref{eq:pure_constr_dual} is the following SDP:
\begin{equation}
\label{eq:pure_constr_primal}
\begin{aligned}
\inf_{\substack{L : \skinnyS_{2 d} \to \R \\ L \emph{ linear}}} \quad  & L(a)  \\	
\emph{s.t.} 
\quad & (\M_d(L))_{u,v} = (\M_d(L))_{w,z}  \,, \quad \text{whenever } \sig(u^\star v) =  \sig( w^\star z) \,, \\
\quad & (\M_d(L))_{1,1} = 1 \,, \\
\quad & \M_{d - d_c}(c \, L) \succeq 0  \,,  \quad \text{for all }  c \in \{1\}\cup C \text{ with }d_c\le d \,.
\end{aligned}
\end{equation}
\end{lemma}

\begin{proof}
Let us denote by $(\cM(C)_d)^{\vee}$ the dual cone of $\cM(C)_d$.
By \eqref{eq:truncated}, one has
$$
(\cM(C)_d)^{\vee} =  \big\{  L : \skinnyS_{2d} \to \R \colon L  \text{ linear},  L(fcf^\star) \geq 0 \text{ for all } c \in \{1\}\cup C, \,f \in \fatS\!_{d - d_c} \big\}\,.
$$
By using a standard Lagrange duality approach, we obtain the dual of SDP \eqref{eq:pure_constr_dual}:
\begin{align}
a_{\min,d}  = &\sup_{a - m \in \cM(C)_d} m = \sup_{m} \inf_{L \in (\cM(C)_d)^{\vee} } (m + L(a -m)) \label{eq:pdfirst}\\
\leq & \inf_{L \in (\cM(C)_d)^{\vee} } \sup_{m}   \, (m + L(a -m)) \label{ineq:pdsecond}\\
= & \inf_{L \in (\cM(C)_d)^{\vee} } (L(a) +  \sup_{m} m (1 - L(1))) \label{eq:pdthird}\\
= & \inf_L \{L(a) \colon L \in (\cM(C)_d)^{\vee},\, L(1) = 1 \} \label{eq:pdlast}\,,
\end{align}
The second equality in \eqref{eq:pdfirst} comes from the fact that the inner minimization problem gives minimal value 0 if and only if $a - m \in \cM(C)_d$.
The inequality in \eqref{ineq:pdsecond}  trivially holds.
The inner maximization problem in \eqref{eq:pdthird} is bounded with maximum value 0 if and only $L(1) = 1$.
Eventually, \eqref{eq:pdlast} is equivalent to SDP~\eqref{eq:pure_constr_primal} by Remark
\ref{rk:Hankelbij} and Lemma \ref{lemma:pureHankel}.
\end{proof}

To establish strong duality for the SDP pair \eqref{eq:pure_constr_dual} and \eqref{eq:pure_constr_primal}, we require a stronger version of algebraic boundedness.

\begin{lemma}\label{l:L0}
Let $C\subset \fatS$ be balanced, and assume $N-x_1^2-\cdots-x_n^2$ for some $N>0$
is a conic combination of elements in $C\cap\RX_2$ and hermitian squares of elements in $\RX_1$.
If $L\in (\cM(C)_d)^\vee$ and $L(1)=0$ then $L=0$.
\end{lemma}

\begin{proof}
For $k\in\N$ let $K_k$ be the convex hull of
$$\Big\{
f_0f_0^\star, \sig(f_0f_0^\star),f_1(N-x_j^2)f_1^\star,\sig(f_1(N-x_j^2)f_1^\star)\colon f_i\in\fatS,\,\deg f_i\le k-i,\, 1\le j\le n
\Big\}.$$
We start by showing that for every $\fatS$-word $w$,
\begin{equation}\label{e:tedious}
N^{\deg w}-ww^\star \in K_{\deg w}.
\end{equation}
We proceed by induction on $\deg w$.
Firstly, $N-x_j^2\in K_1$, and then $N-\sig(x_j)^2=\sig(N-x_j^2)+\sig((x_j-\sig(x_j))^2) \in K_1$.
Now suppose \eqref{e:tedious} holds for $\fatS$-words of degree at most $k$. If $w$ is a $\fatS$-word of degree $k+1$, there are three (partially overlapping) possibilities: $w=x_jv$ for a $\fatS$-word $v$ of degree $k$, in which case
$$N^{\deg w}-ww^\star
=N^{\deg v}(N-x_j^2)+x_j(N^{\deg v}-vv^\star)x_j 
\in K_{k+1};$$
or $w=\sig(x_j)v$ for a $\fatS$-word $v$ of degree $k$, in which case
$$N^{\deg w}-ww^\star
=N^{\deg v}(N-\sig(x_j)^2)+\sig(x_j)^2(N^{\deg v}-vv^\star) 
\in K_{k+1};$$
or $w=\sig(x_jv)$ for a $\fatS$-word $v$ of degree $k$, in which case
$$N^{\deg w}-ww^\star
=N^{\deg w}-\sig((x_jv)(x_jv)^\star)
+\sig((x_jv-\sig(x_jv))(x_jv-\sig(x_jv))^\star)
\in K_{k+1}.$$

Now assume $L\in (\cM(C)_d)^\vee$ and $L(1)=0$. Since $\sig(K_{2d})\subseteq \cM(C)_d$ by the assumption on $C$, we have $L(\sig(ww^\star))=0$ for all $\fatS$-words $w$ of degree at most $d$ by \eqref{e:tedious}.
Then $\sig((u\pm v^\star)(u\pm v^\star)^\star)\in \cM(C)_d$ implies $L(\sig(uv))=0$ for all $\fatS$-words $u,v$ of degree at most $d$. Since $\skinnyS_{2d}$ is spanned by such $\sig(uv)$, we conclude $L=0$.
\end{proof}

\begin{proposition}
\label{p:nogap}
Let $C\subset \fatS$ be balanced, and assume $N-x_1^2-\cdots-x_n^2$ for some $N>0$
is a conic combination of elements in $C\cap\RX_2$ and hermitian squares of elements in $\RX_1$; 
e.g., $N-x_1^2-\cdots-x_n^2\in C$.
Then SDP \eqref{eq:pure_constr_dual} satisfies strong duality, i.e., there is no duality gap between SDP \eqref{eq:pure_constr_primal} and SDP \eqref{eq:pure_constr_dual}. 
\end{proposition}

\begin{proof}
Suppose SDP \eqref{eq:pure_constr_dual} is feasible.
Then $a-a_{\min,d}$ is a boundary point of the cone $\cM(C)_d$ in $\skinnyS_{2d}$. Therefore there is a supporting hyperplane for $\cM(C)_d$ through $a-a_{\min,d}$. In other words, there is a nonzero linear functional $L\in (\cM(C)_d)^\vee$ such that 
$L(a-a_{\min,d})=0$. By Lemma \ref{l:L0} we have $L(1)>0$. After rescaling we can then assume that $L(1)=1$, and so $L(a-a_{\min,d})=0$ implies 
$L(a) = a_{\min,d}$. 
Therefore there is no duality gap.
\end{proof}

\begin{remark}
The condition on the constraint set $C$ in Proposition \ref{p:nogap} is stronger that algebraic boundedness.
Nevertheless, it is satisfied in many prominent instances, for example if $C$ contains $\pm (x_j^2-1)$ for all $j$ (optimization over binary observables) as in Section \ref{sec:bell} below, or if $C$ contains $\pm (x_j^2-x_j)$ for all $j$ (optimization over projections) as in Section \ref{sec:networks}.
Furthermore, if $C$ is algebraically bounded, or $\cD_C^\infty$ is known to be bounded in operator norm, then we can simply add the desired constraint directly to $C$, thus fulfilling the assumption of Proposition \ref{p:nogap} without affecting $\cD_C^\infty$.
\end{remark}

\subsection{Minimizer extraction}
\label{sec:gns}
The goal of this subsection is to derive an algorithm to extract minimizers and certify exactness of \state polynomial optimization problems. 
The forthcoming statements can be seen as \state variants of the results derived in the context of commutative polynomials \cite{curto1998flat}, eigenvalue optimization of noncommutative  polynomials \cite{pironio2010convergent,burgdorf16}, and optimization of trace polynomials~\cite{KMV,nctrace}.

\begin{definition}
\label{def:flatextension}
Suppose $L : \skinnyS_{2d + 2 \delta} \to \R$ is a linear functional. 
We associate to $L$ and $L|_{\skinnyS_{2d}}$ the Hankel matrices $\M_{d+\delta}(L)$ and $\M_d(L|_{\skinnyS_{2d}})$ respectively, and get the block form
\[
\M_{d + \delta}(L) = \begin{bmatrix}
\M_d(L|_{\skinnyS_{2d}}) & B \\[1mm]
B^T & B'
\end{bmatrix} \,.
\]
We say that $L$ is \emph{$\delta$-flat} or that $L$ is a \emph{$\delta$-flat extension} of $L|_{\skinnyS_{2d}}$, if $\M_{d+\delta}(L)$ is flat over $\M_d(L|_{\skinnyS_{2d}})$, i.e., if $\rank \M_{d+\delta}(L) = \rank \M_d(L|_{\skinnyS_{2d}}) $.
\end{definition}

Suppose $L$ is $\delta$-flat and let $r := \rank \M_{d}(L|_{\skinnyS_{2d}}) = \M_{d + \delta}(L)$.
Since $\M_{d+\delta}(L) \succeq 0$, 
we obtain the Gram matrix decomposition 
$ \M_{d + \delta}(L) = [\langle \u , \mathbf{v} \rangle]_{u,v}$ with vectors $\u,\mathbf{v} \in \R^r$, where the labels $u,v$ are $\fatS$-words of degree at most $d + \delta$. 
Then, we define the following $r$-dimensional Hilbert space  
\[
\cK := \text{span} \, \{\u \colon \deg u \leq d + \delta \} = \text{span} \, \{\u \colon \deg u \leq d \} ,
\]
where the equality is a consequence of the flatness assumption.
Let $\skinnyS_{(\delta)}$ denote the subalgebra of $\skinnyS$ generated by $\sig(w)$ for $w\in \mx_\delta\setminus\{1\}$, and denote $\fatS_{(\delta)}=\skinnyS_{(\delta)}\otimes \RX\subset\fatS$.
Each $p \in \fatS_{(\delta)}$ gives rise to the multiplication operator $X_p$ on $\cK$, which leads to the $\star$-representation $\pi : \fatS_{(\delta)} \to \cB(\cK)$ defined by $\pi(p) = X_p$.
Let $\vb$ be the vector representing 1 in $\cK$; then $L(p) = \langle \pi(p) \vb, \vb \rangle$ for all $p \in \fatS_{(\delta)}$. 
In general, elements of $\pi(\skinnyS_{(\delta)})$ are central in $\pi(\fatS_{(\delta)})$; if they are actually scalar multiples of the identity on $\cK$, then $\pi$ is not just a $\star$-representation, but it respects the state symbol in the sense that $\pi(f)=f(\pi(x_1),\dots,\pi(x_n))$ for every $f\in\fatS_{(\delta)}$. 
This fact applies to our SDP hierarchy as follows.

\begin{proposition}\label{prop:st_flat}
Let $a\in\skinnyS$, suppose that $C\subseteq\fatS$
satisfies the assumptions of Proposition \ref{p:nogap}, and let $d,\delta\in\N$ be such that $d,\delta\ge d_a,d_c$ for $c\in C$.
Assume that $L$ is a $\delta$-flat optimal solution of SDP~\eqref{eq:pure_constr_primal}, and let the $\star$-representation
$\pi : \fatS_{(\delta)} \to \cB(\cK)$ and the unit vector $\vb\in\cK$ be constructed as above.
If $\pi(\skinnyS_{(\delta)})=\R$, then
\begin{enumerate}[\rm (i)]
\item $(\lambda,\uX)\in\vD_C$ where $\uX=(\pi(x_1),\dots,\pi(x_n))$ 
and $\lambda(Y)=\langle Y\vb,\vb\rangle$;
\item $L(p)=p(\lambda;\uX)$ for all $p\in\skinnyS_{(\delta)}$;
\item $a_{\min,d + \delta} = L(a) = a_{\min}^\infty$.
\end{enumerate}
\end{proposition}

\begin{proof}
As seen in the paragraph before Proposition \ref{prop:st_flat},
the operators $X_i$ are well-defined (thanks to the flatness assumption) and symmetric. After choosing an orthonormal basis of $\cK$ we can view $X_i$ as $r\times r$ symmetric matrices.
Moreover, $L(p)=\lambda(p)$ for all $p\in\skinnyS_{(\delta)}$.
Since $\pi(\skinnyS_{(\delta)})=\R$, we furthermore have 
$L(p)=p(\lambda;\uX)$ for all $p\in\skinnyS_{(\delta)}$, so (ii) holds.
For $c \in C$, one in particular has $c(\lambda;\uX) = L(c) \geq 0$ 
because $\M_{d - d_c}(c \, L) \succeq 0$ as $L$ is a feasible solution of SDP~\eqref{eq:pure_constr_primal}.
Hence (i) holds. 
Proposition \ref{p:nogap} implies 
$ a_{\min,d + \delta} = L(a) \leq a_{\min}^\infty$; 
on the other hand, $a_{\min}^\infty \leq  a(\lambda;\uX) = L(a)$, 
and therefore (iii) holds.
\end{proof}

\begin{remark}
The condition $\pi(\skinnyS_{(\delta)}) = \R$ in Proposition \ref{prop:st_flat} in particular holds if $L$ is an extreme optimal solution of \eqref{eq:pure_constr_primal}.	
In practice modern SDP solvers rely on interior-point methods using the so-called ``self-dual embedding'' technique \cite[Chapter 5]{wolkowicz2012handbook}. 
Therefore, they will always converge towards an optimum solution of maximum rank; 
see \cite[\S4.4.1]{lasserre2008semidefinite} for more  details. 
In order to obtain a posteriori an extreme linear functional, a commonly used heuristic to minimize the rank is to minimize the trace of the moment matrix over the same constraints involved in  SDP~\eqref{eq:pure_constr_primal} together with the additional constraint $a_{\min,d + \delta} = L(a)$. 
\end{remark}

\subsection{Reduction by exploiting sparsity}\label{ssec:sparse}
In this subsection, we briefly introduce the approaches for reducing the sizes of SDP \eqref{eq:pure_constr_primal} and SDP \eqref{eq:pure_constr_dual} by exploiting sparsity encoded in the state polynomial optimization problem, which are adapted from the case of eigenvalue and trace optimization over nc polynomials \cite{klep2022sparse,wang2021exploiting}. 

\subsubsection{Correlative sparsity}
For $I\subseteq[n]:=\{1,\ldots,n\}$, let $\skinnyS_{I}\subseteq\skinnyS$ (resp.~$\fatS_{I}\subseteq\fatS$) be the subset of \state
(resp.~\ncstate) polynomials in variables $x_i,i\in I$ only.
Let $I_1,\ldots,I_\ell\subseteq[n]$ be a tuple of index sets and further $J_1,\ldots,J_\ell\subseteq C$ be a partition of the constraint polynomials in $C$ such that
\begin{align}
\label{eq:csp1}
 a\in\skinnyS_{I_1}+\cdots+\skinnyS_{I_\ell}; \\
\label{eq:csp2}
J_k\subseteq\fatS_{I_k} \text{ for } k=1,\ldots,\ell.
\end{align}
The tuple of index sets $I_1,\ldots,I_\ell$ is then called the \emph{correlative sparsity pattern} of \eqref{eq:pure_constr} and \eqref{eq:pure_constr1}. We build the Hankel submatrix $\M_d^{I_k}(L)$ (resp.~the localizing submatrix $\M_{d-d_c}^{I_k}(c\, L)$) with respect to the correlative sparsity pattern by retaining only rows and columns indexed by $\W^{\fatS_{I_k}}_d$ (resp. $\W^{\fatS_{I_k}}_{d-d_c}$) for each $k\in[\ell]$ (resp.~each $c\in J_k$).

Let us consider the correlative sparsity adapted version of \eqref{eq:pure_constr_primal}:
\begin{equation}
\label{eq:pure_constr_primal_cs}
\begin{aligned}
\inf_{\substack{L : \skinnyS_{2 d} \to \R \\ L \emph{ linear}}} \quad  & L(a)  \\	
\rm{s.t.} 
\quad & (\M_d^{I_k}(L))_{u,v} = (\M_d^{I_k}(L))_{w,z}  \,, \quad\text{whenever } \sig(u^\star v) =  \sig( w^\star z), \text{ for }k\in[\ell]\,, \\
\quad & (\M_d(L))_{1,1} = 1 \,, \\
\quad & \M_{d - d_c}^{I_k}(c \, L) \succeq 0  \,,  \quad \text{for all }  c \in \{1\}\cup J_k \text{ with }d_c\le d \text{ and }k\in[\ell]\,,
\end{aligned}
\end{equation}
with optimum denoted by $a_{\min,d}^{\cs}$.

\begin{theorem}\label{t:sparse}
Let $C\subseteq\fatS$ be balanced and algebraically bounded, and let $a \in \skinnyS$. 
Suppose that there exist subsets $I_1,\ldots,I_\ell$ and  $J_1,\ldots,J_\ell$ such that \eqref{eq:csp1} and \eqref{eq:csp2} hold,
and that $I_1,\ldots,I_\ell$ satisfy the running intersection property (RIP), i.e., for every $k\in[\ell-1]$, we have that
\begin{align}\label{eq:RIP}
\left(I_{k+1} \cap \bigcup_{j \leq k} I_j \right)\subseteq I_i \quad \text{for some } i \leq k.
\end{align}
Then $\lim_{d\rightarrow\infty}a_{\min,d}^{\cs}=a_{\min}^{\infty}$.
\end{theorem}
\begin{proof}
Let us define 
\[\QM(C^\sig)^{\cs} := \QM(J_1^\sig) + \dots + \QM(J_\ell^\sig).\]
To obtain the convergence result, we need to prove a sparse analog of Theorem \ref{thm:nocyc}, namely 
$a(\lambda;\uX)\geq 0$ for all $(\lambda;\uX) \in \vD_C^\infty$ implies $a + \varepsilon \in \QM(C^\sig)^{\cs}$, for all $\varepsilon > 0$.
This sparse representation result requires an adaptation of 
\cite[Theorem~3.3]{klep2022sparse}, so we only sketch the main steps while emphasizing changes required. 
The proof is by contradiction, so we suppose that $a + \varepsilon \not\in \QM(C^\sig)^{\cs}$ for some $\varepsilon > 0$. 
By the algebraically bounded assumption, 
each quadratic module $\QM(J_k^\sig)$ in $\skinnyS_{I_k}$ is archimedean (see Lemma \ref{l:arch}). In
particular,
$1$ is
an algebraic interior point of 
$\QM(C^\sig)^{\cs}$ in
$\skinnyS_{I_1}+\cdots+\skinnyS_{I_\ell}$. Thus
by the Eidelheit-Kakutani separation theorem (a version of the Hahn-Banach separation theorem suitable for this context; see \cite[Corollary~III.1.7]{Bar02}), 
there exists a unital linear functional $\varphi : \skinnyS_{I_1}+\cdots+\skinnyS_{I_\ell} \to \R$ with $\varphi(\QM(C^\sig)^{\cs}) \subseteq \R_{\geq 0}$ and $\varphi(a) < 0$. We pick any extension of $\varphi$ to a linear functional on $\skinnyS$ which we again denote by $\varphi$.
Let $L: \RX \to \R$ be the unital $\star$-functional given by $L(p) := \varphi (\sig(p))$, and let us denote by $L^k$ the restriction of $L$ to $\RXk$.

Then we proceed exactly as in the proof of Theorem \ref{thm:nocyc} for each $k \in \{1,\ldots,\ell \}$.
Since each quadratic module $\QM(J_k^\sig)\subseteq \skinnyS_{I_k}$ is archimedean, one can apply Proposition \ref{prop:moment} to obtain a vector state $\lambda_k \in\cS(\cH_k)$ and a tuple of self-adjoint operators $\uX^k\in\cB(\cH_k)^{|I_k|}$ such that 
$L^k(p)=\lambda_k(p(\uX^k))$ for all $p\in\RXk$, and $(\lambda_k, \uX^k) \in \vD_{J_k}^\infty$.
Here $\cH_k$ denotes the Hilbert space completion of the quotient of $\RXk$ by the set of nullvectors corresponding to $L^k$, obtained through the GNS construction.
Now, the proof proceeds by induction on $\ell$ to show that there exist a vector state $\lambda$ and a tuple of self-adjoint operators $\uX$ such that 
$b(\lambda,\uX)=\varphi (b)$ for all $b  \in \skinnyS_{I_1}+\dots+\skinnyS_{I_\ell}$, and $(\lambda, \uX) \in \vD_{C}^\infty$.

We focus specifically on the case $\ell = 2$, as the general case then follows by an  inductive argument relying on the running intersection property, similarly to the proof of \cite[Theorem~3.3]{klep2022sparse}. 
We denote by $L^{12}$ the restriction of $L^1$ to $\RXonetwo$ and again we apply Proposition \ref{prop:moment} to obtain a vector state $\lambda_{12} \in\cS(\cH_{12})$ and self-adjoint $\uX^{12}\in \cB(\cH_{12})^{|I_1\cap I_2|}$ such that 
$L^{12}(p)=\lambda_{12}(p(\uX^{12}))$, for all $p \in \RXonetwo$.
For $k \in \{1,2\}$, we denote by $i_k$  the canonical embedding from $\RXonetwo$ to $\RXk$.
Let $\iota_k$ be the canonical embedding from $\cB(\cH_{12})$ to $\cB(\cH_{k})$, satisfying $\iota_k(X_i^{12}) = X_i^k$ for all $i \in I_1\cap I_2$.
Then we apply \cite[Theorem~3.1]{klep2022sparse} to obtain an amalgamation $\cA$ with state $\lambda$ and homomorphisms $j_k : \cB(\cH_k) \to \cA$ such that $j_1 \circ \iota_1 = j_2 \circ \iota_2$.
After performing the GNS construction with $(\cA, \lambda)$, we obtain a Hilbert space $\cK$, a representation $\pi : \cA \to \cB(\cK)$ and a vector $\xi \in \cK$ so that $\lambda(b) = \langle \pi(b) \xi, \xi \rangle$. 
We next define $\uX := (X_1,\dots, X_n)$, with $X_i := \pi( j_1(X^1_i) )$ if $i \in I_1$ and  $X_i := \pi( j_2(X^2_i) )$ otherwise. 
This tuple of operators is well-defined thanks to the amalgamation property.

Let us now set $\tilde L (p) := \langle p(\uX) \xi, \xi \rangle$, for all $p \in \RX$.
We prove that $\tilde L$ agrees with $L^k$ (as well as $L$) on $\RXk$ thanks to the amalgamation setup.
Indeed, for all $p \in \RXk$ one has 
\[
\begin{split}
\tilde L(p)&=\langle p(\uX) \xi, \xi \rangle
= \langle p(\pi(j_k(\uX^k))) \xi, \xi \rangle
= \langle \pi(p(j_k(\uX^k))) \xi, \xi \rangle \\
& = \lambda(p(j_k(\uX^k))) = 
\lambda(j_k(p(\uX^k))) =
\lambda_k (p(\uX^k)) = L^k(p) = L(p).
\end{split}
\]
Hence, this yields $\varphi(b) =  b(\lambda,\uX)$ for all $b \in \skinnyS_k$ and by linearity of $\varphi$, we obtain $\varphi(b) =  b(\lambda,\uX)$ for all $b \in \skinnyS_1+\skinnyS_2$. 
To show that $(\lambda, \uX) \in \vD_{C}^\infty$, we proceed as in the proof of Theorem~\ref{thm:nocyc}.
\end{proof}

\subsubsection{Sign symmetry}\label{sec:ss}
For $a\in\fatS$ and a binary vector $s\in\{0,1\}^{n}$, let $[a]_s\in\fatS$ be defined by $[a]_s(x_1,\ldots,x_n):=a((-1)^{s_1}x_1,\ldots,(-1)^{s_n}x_n)$.
Then $a$ is said to have the sign symmetry represented by a binary vector $s\in\{0,1\}^{n}$ if $[a]_s=a$. We use $S(a)\subseteq\{0,1\}^{n}$ to denote all sign symmetries of $a$ and let $S(C):=\cap_{c\in C}S(c)$ for $C\subseteq\fatS$. 
Denote
\begin{equation*}
    \cU=\{u\in \W_{2d}^{\fatS}\colon S(\{a\}\cup C)\subseteq S(u)\},
\end{equation*}
and let $\skinnyS_\cU\subseteq\skinnyS_{2d}$ be the span of $\{\sig(u)\colon u\in\cU\}$.
Consider the optimization problem given by \eqref{eq:pure_constr1}. We can build a block-diagonal SDP hierarchy for \eqref{eq:pure_constr1} by exploiting its sign symmetries. To this end, we define an equivalence relation $\sim$ on $\W^{\fatS}_d$ by
\begin{equation}\label{eq:equivalence}
    u\sim v \iff u^\star v \in \mathcal{U}.
\end{equation}
The equivalence relation $\sim$ gives rise to a partition of $\W^{\fatS}_d$:
\begin{equation}
    \W^{\fatS}_d=\bigsqcup_{i=1}^{p_d} \W^{\fatS}_{d, i}.
\end{equation}
We build the Hankel submatrix $\M_{d,i}(L)$ (resp.~the localizing submatrix $\M_{d-d_c,i}(c\, L)$) with respect to the sign symmetry by retaining only those rows and columns that are indexed by $\W^{\fatS}_{d, i}$ (resp. $\W^{\fatS}_{d-d_c,i}$) for each $i\in[p_d]$ (resp.~$i\in[p_{d_c}]$).

Let us consider the sign symmetry adapted version of \eqref{eq:pure_constr_primal}:
\begin{equation}
\label{eq:pure_constr_primal_ss}
\begin{aligned}
\inf_{\substack{L : \skinnyS_\cU \to \R \\ L \emph{ linear}}} \quad  & L(a)  \\	
\rm{s.t.} 
\quad & (\M_{d,i}(L))_{u,v} = (\M_{d,i}(L))_{w,z}  \,, \quad\text{whenever } \sig(u^\star v) =  \sig( w^\star z), \text{ for }i\in[p_d]\,, \\
\quad & L(1) = 1 \,, \\
\quad & \M_{d - d_c,i}(c \, L) \succeq 0  \,,  \quad \text{for all }  c \in \{1\}\cup C \text{ with }d_c\le d \text{ and }i\in[p_{d_c}]\,,
\end{aligned}
\end{equation}
with optimum denoted by $a_{\min,d}^{\rm{ss}}$.

\begin{theorem}\label{t:sign}
We have that
$a_{\min,d}^{\rm{ss}}=a_{\min,d}$.
\end{theorem}

\begin{proof}
For a linear functional $L : \skinnyS_{2 d} \to \R$ and $s\in\{0,1\}^{n}$, let 
$L^s : \skinnyS_{2 d} \to \R$ be another linear functional given by $L^s(u)=L([u]_s)$. Suppose that $L$ is an optimal solution of \eqref{eq:pure_constr_primal} and let $L'=\frac{1}{|S(\{a\}\cup C)|}\sum_{s\in S(\{a\}\cup C)}L^s$ which is also an optimal solution of \eqref{eq:pure_constr_primal}. We claim that $L'(\sig(u^\star v))=0$ whenever $u\nsim v\in\W^{\fatS}_d$. By \eqref{eq:equivalence}, if $u\nsim v$, then there exists $s'\in S(\{a\}\cup C)$ such that $[u^\star v]_{s'}=-u^\star v$. We then have
\begin{align*}
    L'(\sig(u^\star v))&=\frac{1}{|S(\{a\}\cup C)|}\sum_{s\in S(\{a\}\cup C)}L^s(\sig(u^\star v))=-\frac{1}{|S(\{a\}\cup C)|}\sum_{s\in S(\{a\}\cup C)}L^s(\sig([u^\star v]_{s'}))\\
    &=-\frac{1}{|S(\{a\}\cup C)|}\sum_{s\in S(\{a\}\cup C)}L^{s+s'}(\sig(u^\star v))=-\frac{1}{|S(\{a\}\cup C)|}\sum_{s\in S(\{a\}\cup C)}L^s(\sig(u^\star v))\\
    &=-L'(\sig(u^\star v)).
\end{align*}
Thus $L'(\sig(u^\star v))=0$ as desired. From this we see that the restriction of $L'$ to $\skinnyS_\cU$ is a feasible solution of \eqref{eq:pure_constr_primal_ss} and so $a_{\min,d}^{\rm{ss}}\le a_{\min,d}$.

On the other hand, let $L:\skinnyS_\cU\to\R$ be an optimal solution of \eqref{eq:pure_constr_primal_ss}. We define a functional $L':\skinnyS_{2d}\to\R$ as follows:
\begin{equation*}
    L'(\sig(u))=\begin{cases}
    L(\sig(u)),&\text{ if }u\in\cU,\\
    0,&\text{ otherwise}.
    \end{cases}
\end{equation*}
One can easily check that $L'$ is a feasible solution of \eqref{eq:pure_constr_primal}. So $a_{\min,d}^{\rm{ss}}\ge a_{\min,d}$ and it follows $a_{\min,d}^{\rm{ss}}= a_{\min,d}$ as desired.
\end{proof}

\begin{remark}
In nc polynomial optimization, the exploitation of sign symmetries can be extended to an iterative procedure via the more general notion of ``term sparsity'' \cite{wang2021exploiting}. 
Due to the multitude of technical details and overhead involved,
the state polynomial version of term sparsity will be explored elsewhere in the future.
\end{remark}

\subsection{Complexification}\label{ss:cx}

\def\sigr{\sig^{\rm re}}
\def\sigi{\sig^{\rm im}}

In this section we detail how to suitably modify state polynomials, quadratic modules $\QM(C^\sig)$ and Theorem \ref{thm:nocyc} to study state polynomial positivity and optimization on states and bounded operators on complex Hilbert spaces.

Instead of a single formal state symbol $\sig(w)$ for $w\in\mx\setminus\{1\}$, we introduce two formal symbols $\sigr(u)$ for $u\in\mx\setminus\{1\}$ and $\sigi(v)$ for $v\neq v^\star\in\mx\setminus\{1\}$, subject to the relations $\sigr(u^\star)=\sigr(u)$ and $\sigi(v^\star)=-\sigi(v)$. Define
\begin{align*}
\skinnyS^\C &= \C\big[\sigr(u),\sigi(v)\colon u\in\mx\setminus\{1\}, 
v\neq v^\star\in\mx\setminus\{1\}\big],\\
\fatS^\C &= \skinnyS^\C\otimes_{\R}\RX=\skinnyS^\C\otimes_{\C}\CX.
\end{align*}
There is a natural involution $\star$ on $\fatS^\C$ which acts as complex conjugation on $\C$ and fixes the generators $x_j,\sigr(u),\sigi(v)$. Let $\Sym\skinnyS^\C$ and $\Sym\fatS^\C$ denote the real vector spaces of $\star$-fixed elements in $\skinnyS^\C$ and $\fatS^\C$, respectively.
Note that
$$\Sym\skinnyS^\C=\R\big[
\sigr(u),\sigi(v)\colon u\in\mx\setminus\{1\}, 
v\neq v^\star\in\mx\setminus\{1\}\big]=:\skinnyS^\R$$
is an infinitely generated real polynomial ring.
Define the $\skinnyS^\C$-linear unital map $$\sig:\fatS^\C\to\skinnyS^\C$$
determined by $\sig(w)=\sigr(w)+i\sigi(w)$ for $w\in\mx\setminus\{1\}$. Then $\sig(f^\star)=\overline{\sig(f)}$ for $f\in\fatS^\C$.
Furthermore, taking the real or the imaginary part of the coefficients of elements in $\fatS^\C$ gives rise to the $\R$-linear maps
$$\re,\im:\fatS^\C\to \skinnyS^\R\otimes_\R \RX.$$
With a slight abuse of notation, we extend the symbols $\sigr,\sigi$ to $\skinnyS^\R$-linear maps $\fatS^\C\to \skinnyS^\R$ given by $\sigr=\re\circ\,\sig$ and $\sigi=\im\circ\,\sig$. Then $\sigr(f^\star)=\sigr(f)$ and $\sigi(f^\star)=-\sigi(f)$ for $f\in\fatS^\C$. A direct calculation gives the following.

\begin{lemma}\label{l:c2r}
Let $c\in\Sym\fatS^\C$ and $p\in\fatS^\C$. Then
$$\sig(pcp^\star)=
\sigr\big((\re p)c(\re p)^\star\big)
+\sigr\big((\im p)c(\im p)^\star\big)
+\sigi\big((\re p)c(\im p)^\star\big)
-\sigi\big((\im p)c(\re p)^\star\big).
$$
\end{lemma}

Let $C\subseteq\fatS^\C$ be a set of state polynomial constraints, and define
$$\mathcal{D}_C^\infty := \{ (\lambda,\underline{X}) \in \cS(\cH)\times\cB(\cH)^n : 
X_j=X_j^*,\ c(\lambda;\underline{X}) \succeq 0 \ \text{for all}\ c\in C  \}$$
where $\cH$ is a separable complex Hilbert space (which is unique up to an isomorphism). 

\begin{example}
Let $C=\{\pm(1+\sig(x_1x_2x_3)^2)\}$. In the real framework (where real Hilbert spaces and real states are considered), $\cD_C^\infty=\emptyset$.
On the other hand, $\cD_C^\infty\neq\emptyset$ in the complex framework. Concretely, the state $\lambda=\frac12\tr$ on $2\times2$ complex matrices, and hermitian matrices
$$
X_1=\begin{pmatrix}0&1\\1&0\end{pmatrix},\quad
X_2=\begin{pmatrix}2&0\\0&0\end{pmatrix},\quad
X_3=\begin{pmatrix}0&i\\-i&0\end{pmatrix}
$$
satisfy $1+\lambda(X_1X_2X_3)^2=0$.
\end{example}

Suppose $C$ is balanced, in the sense that $C^\star=C$ and $-(C\setminus\Sym\fatS^\C)\subseteq C$. If
$$C'=C\cap\Sym\fatS^\C\cup\left\{
\pm\tfrac12(c+c^\star), 
\pm\tfrac12(ic^\star-ic)
\colon c\in C\setminus\Sym\fatS^\C
\right\}
\subseteq\Sym\fatS^\C,$$
then $\mathcal{D}_C^\infty=\mathcal{D}_{C'}^\infty$ since $C$ is balanced. When interested in state polynomial positivity within the complex framework, this observation allows one to replace balanced subsets of $\fatS^\C$ with the subsets of $\Sym\fatS^\C$. For the rest of this section we therefore restrict to constraint sets $C\subset\Sym\fatS^\C$. 

Analogously as before, we say that $C\subseteq\Sym\fatS^\C$ is algebraically bounded if there exists $N>0$ such that $N-x_1^2-\cdots-x_n^2=\sum_ip_ic_ip_i^\star$ for some $c_i\in\{1\} \cup C\cap \CX$ and $p_i\in \CX$.
Furthermore, set
$$C^\sig=
\{\sig(pcp^\star)\colon p\in\CX,c\in\{1\}\cup C\}
\subseteq\skinnyS^\R.
$$
One can alternatively introduce $C^\sig$ using only real polynomials according to Lemma \ref{l:c2r}.

\begin{lemma}\label{l:archcx}
If $C\subseteq\Sym\fatS^\C$ is algebraically bounded then the quadratic module $\QM(C^\sig)\subseteq \skinnyS^\R$ is archimedean.
\end{lemma}

\begin{proof}
It suffices to extend the proof of Lemma \ref{l:arch}, and show that the generators $\sigr(w),\sigi(w)$ of $\skinnyS^\R$ are bounded with respect to $\QM(C^\sig)$. As in the proof of Lemma \ref{l:arch} we see that for $w\in\mx$ there exists $m>0$ such that $m-\sig(ww^\star)\in\QM(C^\sig)$.
Then
\begin{equation*}
\begin{split}
\tfrac14+m\pm\sigr(w) &= 
\sig\left((\tfrac12\pm w)(\tfrac12\pm w)^\star\right)
+m-\sig(ww^\star) \in \QM(C^\sig),\\
\tfrac14+m\mp\sigi(w) &= 
\sig\left((\tfrac{i}{2}\pm w)(\tfrac{i}{2}\pm w)^\star\right)
+m-\sig(ww^\star) \in \QM(C^\sig).\qedhere
\end{split}
\end{equation*}
\end{proof}

The following is the complex analog of Theorem \ref{thm:nocyc}.

\begin{theorem}\label{thm:nocyccx}
Let $C\subseteq \Sym\fatS^\C$ be algebraically bounded.
Then for $a\in\skinnyS^\R$ the following are equivalent:
\begin{enumerate}[\rm (i)]
\item$a(\lambda;\uX)\geq 0$ for all 
$(\lambda;\uX) \in \cD_C^\infty$;\vspace{0.25ex}
\item$a+\varepsilon \in \QM(C^\sig)$ for all $\varepsilon>0$.
\end{enumerate}
\end{theorem}

\begin{proof}
The implication (ii)$\Rightarrow$(i) is straightforward as in the proof of Theorem \ref{thm:nocyc}.
For the implication (i)$\Rightarrow$(ii), suppose $a+\varepsilon \notin \QM(C^\sig)$ for some $\varepsilon>0$. By Lemma \ref{l:archcx} and Proposition \ref{prop:KD}, there exists a unital homomorphism $\varphi: \skinnyS^\R\to\R$ with $\varphi(\QM(C^\sig))\subseteq \R_{\ge0}$ and $\varphi(a)<0$. Hence
$$\varphi(\sig (pp^\star))\ge0,\qquad \varphi(\sig (p(N-x_1^2-\cdots-x_n^2)p^\star))\ge0$$
for all $p\in\CX$. 
Consider the unital $\star$-functional $L:\CX\to\C$ given by $L(p)=\varphi(\sigr(p))+i\varphi(\sigi(p))$.
By a complex version of Proposition \ref{prop:moment} (which is obtained from the GNS construction on $\CX$), there exist a vector state $\lambda\in\cS(\cH)$ and $\uX=\uX^*\in\cB(\cH)^n$ such that 
$L(p)=\lambda(p(\uX))$ for all $p\in\CX$.
Therefore $\varphi(b)=b(\lambda;\uX)$ for all $b\in \skinnyS^\R$. The rest follows as in the proof of Theorem \ref{thm:nocyc}.
\end{proof}

Let $C\subseteq\Sym\fatS^\C$ be algebraically bounded, and $a\in\skinnyS^\R$. For $d\in\N$ let
$$\cM(C)_d:=\left\{
\sum_{i=1}^{K} \sig(f_ic_if_i^\star)\colon K\in\N,\ f_i\in\fatS^\C,\ c_i \in \{1\}\cup C,\ \deg(f_ic_if_i^\star)\le 2d
\right\}.$$
Then the SDP hierarchy
\begin{equation}\label{e:cx_dual}
\sup \{ m \colon a - m \in \cM(C)_d \}
\end{equation}
converges to $\inf_{\cD^\infty_C}a$ from above. Analogously to Lemma \ref{lemma:pure_dual}, the dual of \eqref{e:cx_dual} is constructed using Hankel matrices, and there is no duality gap if the assumption of Proposition \ref{p:nogap} is satisfied.

More concretely, let $L:\skinnyS^\R\to\R$ be a unital functional, and $c\in\Sym\fatS^\C$. The hermitian localizing Hankel matrix $\M_d^\C(c\, L)$ is indexed by $\fatS^\C$-words of degree at most $d$ (here, a $\fatS^\C$ word is a product of $\sigr(u),\sigi(v),x_j$), and its $(u,v)$-entry equals $L(\sigr(u^\star cv))+i L(\sigi(u^\star cv))$. Then
\begin{equation}
\label{e:cx_primal}
\begin{aligned}
\inf_{\substack{L : \skinnyS_{2 d}^\R \to \R \\ L \emph{ linear, }L(1)=1}} \quad  & L(a)	
\qquad
\rm{s.t.} 
\quad \M_{d - d_c}^\C(c \, L) \succeq 0  \,,  \quad \text{for all }  c \in \{1\}\cup C \text{ with }d_c\le d
\end{aligned}
\end{equation}
is the dual of \eqref{e:cx_dual}. 
Alternatively, to reformulate \eqref{e:cx_primal} as an SDP over real numbers, we refer the reader to \cite{wang2023efficient}.


\section{Nonlinear Bell inequalities}
\label{sec:bell}

In this section we connect \state polynomial optimization to violations of nonlinear Bell inequalities, establish a further reduction of our optimization procedures based on conditional expectation that is tailored to the quantum-mechanical formalism (Proposition \ref{p:condexp}), and outline a few examples. 

For the sake of simplicity, we restrict to bipartite models where two parties share a state and use binary observables to produce measurements.
A ($m$-input 2-output) \emph{quantum commuting model} is then given as a triple $(\lambda,\uA,\uB)$ where $\lambda\in\cS(\cH)$ is a state and 
$\uA=(A_1,\dots,A_m),\uB=(B_1,\dots,B_m)$ are commuting tuples of binary observables in $\cB(\cH)$:
$$A_i^*=A_i,\quad A_i^2=I,\quad 
B_j^*=B_j,\quad B_j^2=I,\quad A_iB_j=B_jA_i$$
for all $1\le i,j\le m$. The correlations produced by $(\lambda,\uA,\uB)$ are determined by $\lambda(A_iB_j)$ for $i,j=0,\dots m$ where $A_0=B_0=I$.
If $\cH=\cH'\otimes \cH'$ for a finite-dimensional $\cH'$ and $A_i=A_i'\otimes I$, $B_j=I\otimes B_j'$ then $(\lambda,\uA,\uB)$ is a (finite-dimensional)
\emph{spatial quantum model}. 
In this case, $\dim\cH'$ is the local dimension of $(\lambda,\uA,\uB)$, and $\lambda$ is usually given by a density matrix.
On the other hand, if $\uA$ and $\uB$ are tuples of commuting operators, then $(\lambda,\uA,\uB)$ is \emph{classical}.
In this case, the correlations can be obtained as expectations of products of binary random variables on a probability space.

To warm up, consider the expression
\begin{equation}\label{e:bell1}
\lambda(A_1B_1)+\lambda(A_1B_2)+\lambda(A_2B_1)-\lambda(A_2B_2)
\end{equation}
for a model $(\lambda,\uA,\uB)$.
The classical Bell inequality states that \eqref{e:bell1}
is at most $2$
for classical models.
On the other hand, \eqref{e:bell1} attains the value $2\sqrt{2}$ for a spatial quantum model with local dimension $2$.
Furthermore, Tsirelson's bound \cite{tsirelson} implies that the value $2\sqrt{2}$ is optimal for all quantum commuting models.
From the perspective of this paper, Tsirelson's bound can be recovered as a \state polynomial optimization problem
\begin{equation*}
\begin{split}
\sup\ & \sig(x_1y_1)+\sig(x_1y_2)+\sig(x_2y_1)-\sig(x_2y_2) \\
&\ \text{ s.t. }\ x_i^2=1,\,y_j^2=1,\,[x_i,y_j]=0.
\end{split}
\end{equation*}
 
Upper bounds on quantum violations of linear Bell inequalities can be found using the NPA hierarchy \cite{navascues2008convergent} for eigenvalue optimization of noncommutative polynomials; for example, one can get Tsirelson's bound on violations of \eqref{e:bell1} by eigenvalue-optimizing $x_1y_1+x_1y_2+x_2y_1-x_2y_2$ subject to $x_j^2=y_j^2=1$ and $[x_i,y_j]=0$.
On the other hand, covariance of quantum correlations \cite{PHBB} and detection of partial separability \cite{Uffink} lead to more general {\em polynomial} Bell inequalities. While linear Bell inequalities are linear in expectation values of (products of) observables, polynomial Bell inequalities contain multivariate polynomials in expectation values of (products of) observables. 
Even for classical models, nonlinearity complicates the study of polynomial Bell inequalities; for example, the supremum of a Bell-like expression over classical models can be strictly larger than the supremum over deterministic models.
Nonlinearity also renders noncommutative polynomial eigenvalue optimization, which is commonly used to bound quantum violations of linear Bell inequalities, inapplicable to polynomial Bell inequalities.
On the other hand, \state polynomial optimization gives upper bounds on violations of polynomial Bell inequalities. 

\subsection{Universal algebras of binary observables}
\label{sec:uni}

In this section we derive further simplifications for optimization of a \state polynomial subject to a balanced constraint set of the form
\begin{equation}\label{e:specialC}
C=\left\{\pm(1-x_1^2),\cdots,\pm(1-x_n^2),\pm[x_{i_1},x_{j_1}],\dots, \pm[x_{i_\ell},x_{j_\ell}]
\right\}
\end{equation}
for some $1\le i_k,j_k\le n$. 
As mentioned above, optimal Bell inequalities correspond to optimization problems subject to constraint sets of the form \eqref{e:specialC}.
By Corollary \ref{cor:pure_cvg}, every state polynomial optimization problem on \eqref{e:specialC} admits a convergent SDP hierarchy as in Section \ref{sec:constrained}, and these SDPs satisfy strong duality by Proposition \ref{p:nogap}.

Let $G$ be a group. Analogously to the construction of \ncstate polynomials, one can define the \emph{state group algebra} $\fatS(G)$ of $G$: namely, let $\skinnyS(G)$ be the real polynomial ring in commutative symbols $\sig(g)$ for $g\in G\setminus\{1\}$, subject to $\sig(g^{-1})=\sig(g)$, and let $\fatS(G)=\skinnyS(G)\otimes \R[G]$, where $\R[G]$ is the real group $\star$-algebra of $G$, where the involution is given by $g^\star = g^{-1}$ for $g\in G$. As before, there is a natural map $\fatS(G)\to\skinnyS(G)$. Let $\Sigma \fatS(G)^2$ denote the set of sums of hermitian squares $ff^\star$ for $f\in\fatS(G)$.

Returning to \eqref{e:specialC}, consider the group
$$G=\langle x_1,\dots,x_n \mid 
x_1^2=\cdots=x_n^2=1,x_{i_1}x_{j_1}=x_{j_1}x_{i_1},\dots, x_{i_\ell}x_{j_\ell}=x_{j_\ell}x_{i_\ell}
\rangle.$$
Let $\pi:\mx\to G$ be the canonical homomorphism. We extend it to a $\sig$-respecting $\star$-homomorphism $\pi:\fatS\to\fatS(G)$. Then for $a\in\skinnyS$,
\begin{equation}\label{e:reduced}
a\in \QM(C^\sig) \qquad\iff\qquad 
\pi(a)\in \sig\left(\Sigma \fatS(G)^2\right).
\end{equation}
The relation \eqref{e:reduced} is advantageous in optimizing \state polynomials subject to $C$: the sizes of SDPs \eqref{eq:pure_constr_dual} and \eqref{eq:pure_constr_primal} can be reduced by indexing with a basis of $\fatS(G)$, and only a single semidefinite constraint is needed (corresponding to $\Sigma \fatS(G)^2$).
This reduction is used in all 
subsequent computational examples.

A further reduction is sometimes possible. Given $f\in\skinnyS(G)$, its \emph{support} are elements of $G$ appearing in $f$.

\begin{proposition}\label{p:condexp}
Let $a\in\skinnyS$, and let $H\subseteq G$ be the subgroup generated by the support of $\pi(a)$. Then
\begin{equation}\label{e:reduced2}
a\in \QM(C^\sig) \qquad\iff\qquad 
\pi(a)\in \sig\left(\Sigma \fatS(H)^2\right).
\end{equation}
\end{proposition}

\begin{proof}
Let ${\bf 1}_H:G\to H\cup\{0\}$ be the indicator function, where ${\bf 1}_H(g)=g$ if $g\in H$ and ${\bf 1}_H(g)=0$ otherwise.
Let $E:\fatS(G)\to\fatS(H)$ be the unital $\star$-linear map given by
$$E\big(\sig(g_1)\cdots\sig(g_\ell)g_0 \big) =
\sig({\bf 1}_H(g_1))\cdots\sig({\bf 1}_H(g_\ell)){\bf 1}_H(g_0).
$$
Note that $E$ commutes with $\sig$, restricts to a ring homomorphism $\skinnyS(G)\to\skinnyS(H)$, and has conditional expectation properties (cf.~\cite[Section 3]{SS13}):
\begin{enumerate}[(i)]
    \item $E(b_1ab_2)=b_1E(a)b_2$ for $a\in\fatS(G)$ and $b_1,b_2\in\fatS(H)$,
    \item $E(\Sigma\fatS(G)^2)=\Sigma\fatS(H)^2$.
\end{enumerate}
The second property follows by \cite[Proposition 3.4]{SS13} and $E:\skinnyS(G)\to\skinnyS(H)$ being a homomorphism. Since $\pi(a)\in\skinnyS(H)$, (ii) implies $\pi(a)\in \sig\left(\Sigma \fatS(H)^2\right)$ if and only if $\pi(a)\in \sig\left(\Sigma \fatS(G)^2\right)$, and the rest follows by \eqref{e:reduced}.
\end{proof}
If $H$ is a proper subgroup of $G$, the sizes of SDPs \eqref{eq:pure_constr_dual} and \eqref{eq:pure_constr_primal} can thus be further decreased by Proposition \ref{p:condexp}; this is illustrated in Example \ref{exa1} below.

\subsection{Examples}\label{sec:exa}

We demonstrate the optimization results from Section \ref{sec:hierarchy} on the following polynomial Bell inequalities. The codes for reproducing these results
over real Hilbert spaces
are available at
\begin{center}
\url{https://github.com/wangjie212/NCTSSOS/blob/master/examples/stateopt.jl}
\end{center}
and the codes for reproducing these results over complex Hilbert spaces are available at
\begin{center}
\url{https://github.com/wangjie212/NCTSSOS/blob/master/examples/complex_state.jl}
\end{center}
For all examples (except Example 8.1.3 with $d=5$), we employ {\tt MOSEK 10.0} as an SDP solver.
For more details on the modeling syntax, we refer the interested programmer to the tutorial from \cite[Appendix~B.2]{sparsebook} that describes a similar syntax to perform trace polynomial optimization. 

\subsubsection{Example}\label{exa1}

One of the first considered polynomial Bell inequalities is
\begin{equation}\label{e:bell2}
\lambda(A_1B_2+A_2B_1)^2+
\lambda(A_1B_1-A_2B_2)^2
\le4
\end{equation}
given in \cite{Uffink}, where it is shown that \eqref{e:bell2} holds for all classical models, and for all spatial quantum models with local dimension $2$ (the equality is obtained for a model with a maximally entangled state). In \cite{NKI}, \eqref{e:bell2} is shown to hold for all spatial quantum models. 
An automatized proof of \eqref{e:bell2} for arbitrary quantum commuting models can be obtained by solving the optimization problem
\begin{equation}\label{e:bell3}
\begin{split}
\sup\ & 
\left(\sig(x_1y_2)+\sig(x_2y_1)\right)^2+
\left(\sig(x_1y_1)-\sig(x_2y_2)\right)^2 \\
&\ \text{ s.t. }\ x_i^2=1,y_j^2=1,\, [x_i,y_j]=0 \text{ for }i,j=1,2.
\end{split}
\end{equation}
Let $C=\{\pm(1-x_j^2),\pm(1-y_j^2),\pm[x_i,y_j]\}_{i,j=1,2}$.
The relaxation of \eqref{e:bell3} with $d=3$ as in Section \ref{sec:constrained},
\begin{equation}\label{e:bell31}
\inf\ \mu \ \text{ s.t. }\ 
\mu-\left(\sig(x_1y_2)+\sig(x_2y_1)\right)^2-
\left(\sig(x_1y_1)-\sig(x_2y_2)\right)^2 \in \mathcal{M}(C)_d
\end{equation}
outputs $4$, which coincides with the classical value in \eqref{e:bell2}. The concrete implementation of \eqref{e:bell31} encodes the relations $x_j^2=y_j^2=1$ and $[x_i,y_j]=0$ directly into the SDP, as in Section \ref{sec:uni}. The resulting SDP has 2032 variables, and a $209\times 209$ semidefinite constraint.

Alternatively, we can also invoke Proposition \ref{p:condexp}. The support of 
$\left(\sig(x_1y_2)+\sig(x_2y_1)\right)^2-
\left(\sig(x_1y_1)-\sig(x_2y_2)\right)^2$
in
$$G=\langle x_i,y_j \mid 
x_i^2=y_j^2=1,x_iy_j=y_jx_i 
\text{ for } i,j=1,2
\rangle$$
is $\{x_iy_j\}_{i,j=1,2}$, which generates the subgroup $H$ of $G$ consisting of all words in generators $x_i,y_j$ of even length.
Cutting down the aforementioned SDP with respect to $H$ then results in an SDP with 933 variables and a $112\times 112$ semidefinite constraint, which returns the value 4 in shorter time.

\subsubsection{Example}\label{exa2}

\def\cov{\operatorname{cov}}

Polynomial Bell inequalities also arise from covariances of quantum correlations. Let
$$\cov_{\lambda}(X,Y) = 
\lambda(XY)-\lambda(X)\lambda(Y).$$
In \cite{PHBB} it is shown that while
\begin{equation}\label{e:bell4}
\begin{split}
&\cov_\lambda(A_1,B_1)+\cov_\lambda(A_1,B_2)
+\cov_\lambda(A_1,B_3) \\
+&\cov_\lambda(A_2,B_1)+\cov_\lambda(A_2,B_2)
-\cov_\lambda(A_2,B_3)\\
+&\cov_\lambda(A_3,B_1)-\cov_\lambda(A_3,B_2)
\end{split}
\end{equation}
is at most $\frac92$ for classical models, it attains the value 5 for a spatial quantum model of local dimension 2 and a maximally entangled state. The authors also performed extensive numerical search within spatial quantum models with local dimension at most $5$, but no higher value of \eqref{e:bell4} was found. They left it as an open question whether higher dimensional entangled states could lead to larger violations \cite[Appendix D.1(b)]{PHBB}.

Let
\begin{align*}
b=\,&\sig(x_1y_1)-\sig(x_1)\sig(y_1)+\sig(x_1y_2)-\sig(x_1)\sig(y_2)+\sig(x_1y_3)-\sig(x_1)\sig(y_3) \\
&+\sig(x_2y_1)-\sig(x_2)\sig(y_1)+\sig(x_2y_2)-\sig(x_2)\sig(y_2)-\sig(x_2y_3)+\sig(x_2)\sig(y_3) \\
&+\sig(x_3y_1)-\sig(x_3)\sig(y_1)-\sig(x_3y_2)+\sig(x_3)\sig(y_2)\,.
\end{align*}
The relaxation of
\begin{equation*}
\sup\ b\ \text{ s.t. }\  x_i^2=1,\,y_j^2=1,\,[x_i,y_j]=0 \text{ for }i,j=1,2,3
\end{equation*}
with $d=2$ returns 5. Therefore the value of \eqref{e:bell4} is at most 5 for all quantum commuting models.

\subsubsection{Example}\label{exa3}
In the previous two examples, the maximal violation of a polynomial Bell inequality was attained at a maximally entangled state.
Next, consider the expression
\begin{equation}\label{e:newexa}
\begin{split}
&\lambda(A_2+B_1 + B_2
- A_1B_1 + A_2 B_1
+ A_1 B_2+ A_2 B_2) \\
& -\lambda(A_1)\lambda(B_1)
-\lambda(A_2)\lambda(B_1)
-\lambda(A_2)\lambda(B_2)
-\lambda(A_1)^2- \lambda(B_2)^2\,.
\end{split}
\end{equation}
Below we show that:
\begin{enumerate}[(i)]
    \item \eqref{e:newexa} is bounded by 3.375 for classical models and spatial quantum models with maximally entangled states, and this bound is obtained by a classical model with a discrete 3-atomic measure;
    \item \eqref{e:newexa} is bounded by 3.51148 for any quantum commuting model, and this bound is obtained by a spatial quantum model of local dimension 2.
\end{enumerate}
Let
\[
\begin{split}
b =\,&\sig(x_2)+\sig(y_1)+\sig(y_2)
-\sig(x_1y_1)+\sig(x_2y_1)+\sig(x_1y_2)+\sig(x_2y_2) \\
& -\sig(x_1)\sig(y_1)
-\sig(x_2)\sig(y_1)-\sig(x_2)\sig(y_2)
-\sig(x_1)^2-\sig(y_2)^2 \,.
\end{split}
\]

(ii): To solve the optimization problem
\begin{equation}\label{e:bell6}
\sup\ b\ \text{ s.t. }\  x_i^2=1,\,y_j^2=1,\,[x_i,y_j]=0 
\text{ for }i,j=1,2,
\end{equation}
we first solve the SDP for the relaxation of \eqref{e:bell6} with $d=2$ 
as in \eqref{eq:pure_constr_primal}.
The output is 3.51148; moreover, the resulting Hankel matrix is flat, and the assumptions of Proposition \ref{prop:st_flat} are satisfied. Therefore we can perform the finite-dimensional GNS construction and extract a $4$-dimensional quantum commuting model attaining 3.51148. After a unitary basis change, the extracted model is evidently spatial quantum, of local dimension 2, and given by
\begin{equation*}
\begin{split}
\ket{\psi}=
\begin{pmatrix}
-\cos\beta\sin\tfrac{\beta}{3} \\
\cos\beta\cos\tfrac{\beta}{3} \\
-\sin\beta\cos\tfrac{2\beta-\alpha}{3} \\
\sin\beta\sin\tfrac{2\beta-\alpha}{3}
\end{pmatrix},
\quad
&
A_1=\begin{pmatrix} 1& 0 \\ 0 & -1
\end{pmatrix}\otimes I,\quad
A_2=
\begin{pmatrix}
\cos\alpha&\sin\alpha \\
 \sin\alpha & -\cos\alpha
\end{pmatrix}
\otimes I,
\\
&
B_1=I\otimes \begin{pmatrix} 1& 0 \\ 0 & -1
\end{pmatrix},\quad
B_2=
-I\otimes \begin{pmatrix}
\cos\alpha&\sin\alpha \\
 \sin\alpha & -\cos\alpha
\end{pmatrix}
\end{split}
\end{equation*}
and $\lambda(Y) = \bra{\psi}Y\ket{\psi}$ for $\alpha=-4.525$ and $\beta=2.192$.

(i): If in the optimization problem \eqref{e:bell6} one restricts only to \emph{tracial} states (i.e., $\lambda\in\cS(\cH)$ satisfying $\lambda(uv)=\lambda(vu)$), then its solution gives an upper bound of \eqref{e:newexa} for both classical models and spatial quantum models with a maximally entangled state.
Using the relaxation with $d=2$ in the SDP hierarchy \cite[Section 5.3]{KMV} for the tracial version of \eqref{e:bell6} one obtains an upper bound $3.375$. Again, the resulting Hankel matrix is flat, so a maximizing $3$-dimensional model with a tracial state can be extracted \cite[Section 5.4]{KMV}. 
In this model, all the operators commute, so the model is classical, on a probability space of size 3. Once the bound on the size of the probability space is known, we can search for a maximizing classical model exactly, resulting in
\begin{equation*}
\rho=\diag\left(\frac14,\frac38,\frac38\right),\quad
A_1=\diag(1,-1,-1),\quad A_2=\diag(1,1,-1),\quad
B_1=I,\quad B_2=A_2
\end{equation*}
and $\lambda(Y) = \tr(\rho Y)$, for which the value of \eqref{e:newexa} is $\frac{27}{8}=3.375$.
Lastly, since $\rho$ has rational entries with denominator 8, the upper bound $\frac{27}{8}$ can also be reached by a spatial quantum model with a maximally entangled state with (possibly non-minimal)
local dimension $8\cdot 3=24$.

To complete the picture, let us mention that the maximum of \eqref{e:newexa} for deterministic models is 2.

\section{Bell inequalities for network scenarios}
\label{sec:networks}

As seen in the previous section, 
a polynomial Bell inequality corresponds to optimizing a \state polynomial subject to noncommutative constraints.
On the other hand, correlation inequalities for general quantum networks \cite{Fri,pozas2019bounding,ligthart21,tavakoli22}
correspond to optimizing a \state polynomial subject to both noncommutative and \state constraints.

Following \cite{Fri}, a \emph{correlation} or \emph{network scenario} is given by
\begin{enumerate}[(1)]
    \item a set $[M]=\{1,\dots,M\}$ of parties;
    \item a set $[S]=\{1,\dots,S\}$ of sources; and
    \item a relation $\rightsquigarrow$ on $[S]\times[M]$, where $s\rightsquigarrow m$ means that the party $m$ has access to the source $s$.
\end{enumerate}
Parties can have several inputs (questions) and outputs (answers);
for $m\in [M]$ let $a_m$ and $b_m$ be the number of inputs and outputs of $m$, respectively.

A \emph{(spatial) quantum} model for such a network is given by
\begin{enumerate}[(i)]
\item (finite-dimensional) Hilbert spaces $\cH_{(s,m)}$ for $s\rightsquigarrow m$;
\item for each $m\in [M]$, projective-valued measures (PVMs) $(P_{m,i,j})_{j=1}^{b_m}$ for $i=1,\dots,a_m$ where
$$P_{m,i,j} \in\cB\left(\bigotimes_{s: s\rightsquigarrow m}\cH_{s,m}\right),\qquad
P_{m,i,j}=P_{m,i,j}^*=P_{m,i,j}^2,\quad
\sum_{j=1}^{b_m}P_{m,i,j}=I;
$$
\item density matrices
$$\rho_s\in\cB\left(\bigotimes_{m: s\rightsquigarrow m}\cH_{s,m}\right)$$
representing states, for $s\in [S]$.
\end{enumerate}
The correlations of this model are
$$p(i_1,j_1,i_2,j_2,\dots,i_{M},j_{M}) = \tr\left(
\bigotimes_{s\in [S]}\rho_s \cdot
\bigotimes_{m\in [M]}P_{m,i_m,j_m}
\right)$$
with a slight abuse of notation, since the tensor factors need to be appropriately ordered.

As in \cite[Section III.C]{ligthart23} (cf. \cite[Section II.C]{ligthart21} and \cite[Definition 3.2]{renou22}), a \emph{reduced quantum model} for such a network is given by
\begin{enumerate}[(i)]
\item a (possibly infinite-dimensional) Hilbert space $\cH$;
\item for each $m\in [M]$, projective-valued measures (PVMs) $(P_{m,i,j})_{j=1}^{b_m}$ for $i=1,\dots,a_m$ where
$$P_{m,i,j} \in\cB(\cH),\qquad
P_{m,i,j}=P_{m,i,j}^*=P_{m,i,j}^2,\quad
\sum_{j=1}^{b_m}P_{m,i,j}=I,
$$
and
$$[P_{m,i,j},P_{m',i',j'}]=0\qquad \text{for } m\neq m';$$
\item a state $\lambda\in\cS(\cH)$ satisfying
\begin{equation}
\label{e:factor}
\lambda(Q_1\cdots Q_\ell)
=\lambda(Q_1)\cdots\lambda(Q_\ell)
\end{equation}
whenever each $Q_k$ is in the algebra generated by the PVMs for $m_k$, and for all $k\neq k'$, $m_k\neq m_{k'}$ and there is no $s\in [S]$ with 
$m_k\leftsquigarrow s\rightsquigarrow  m_{k'}$.
\end{enumerate}
The correlations of this model are
$$p(i_1,j_1,i_2,j_2,\dots,i_{M},j_{M}) = \lambda\left(
P_{1,i_1,j_1}P_{2,i_2,j_2}\cdots
P_{M,i_{M},j_{M}}
\right)$$

Clearly, correlations of spatial quantum models are produced by reduced quantum models. 
When all operators in a reduced quantum model commute (in which case measurements are given by indicator functions on a probability space, and the state is given by the integration with respect to the probability measure),
the model is \emph{classical}.

A (classical/spatial quantum/reduced quantum) polynomial Bell inequality for a network scenario is an upper bound on a polynomial expression in correlations, valid for every (classical/spatial quantum/reduced quantum) model. We can obtain Bell inequalities for reduced quantum models of network scenarios using the SDP hierarchy from Section \ref{sec:hierarchy} as follows.

Consider the description of a network scenario as at the beginning of this section, and let $\mathcal{B}$ be a polynomial expression in correlations $p(i_1,j_1,\dots,i_{M},j_{M})$.
For each $m,i,j$ let $x_{m,i,j}$ be a freely noncommuting self-adjoint variable, and let $b\in\skinnyS$ be the state polynomial obtained from $\mathcal{B}$ by replacing $p(i_1,j_1,\dots,i_{M},j_{M})$ with
$\sig(x_{1,i_1,j_1}\cdots x_{M,i_{M},j_{M}})$.

\begin{corollary}\label{c:network}
Let $\mathcal{B}$ and $b$ be as above.
Then $\beta\in\R$ is the smallest constant such that
$\mathcal{B} \le \beta$ for every reduced quantum model if and only if $\beta$ is the output of the \state polynomial optimization problem
\begin{equation}
\begin{aligned}
\sup\ & b  \\	
\rm{s.t.} 
\quad & x_{m,i,j}^2=x_{m,i,j},\ \sum_{j=1}^{b_m}x_{m,i,j}=1\,, \\
& x_{m,i,j}x_{m,i,j'}=0,\ \text{ for } j\neq j',\ [x_{m,i,j},x_{m',i',j'}]=0 \ \text{ for } m\neq m'\,, \\
& \sig(w_1\cdots w_\ell)=\sig(w_1)\cdots\sig(w_\ell) \ \text{ for } w_k\in\langle x_{m_k,i,j}:i,j\rangle \text{ where }m_k \text{ are distinct }\\
&\hspace{22.5em}\text{and not sharing sources.} \\
\end{aligned}
\end{equation}
\end{corollary}
Note that the constraints $x_{m,i,j}x_{m,i,j'}=0$ for $j\neq j'$ above are redundant, but convenient for reducing the size of SDPs.

In particular, Corollaries \ref{cor:pure_cvg} and \ref{c:network} together yield a convergent SDP hierarchy for optimization of Bell expressions over reduced quantum models of an arbitrary network scenario.
By Proposition \ref{p:nogap},
it is easy to see that in the special case of the bilocal scenario (see Section \ref{sec:biloc} below), this SDP hierarchy is equivalent to the one presented in \cite{pozas2019bounding}, and whose convergence was first proved in \cite{renou22}.
For a different convergent SDP hierarchy based on the quantum de Finetti theorem, see \cite{ligthart21}.

\subsection{Bilocal scenario}\label{sec:biloc}

In the \emph{bilocal} scenario, there are three parties and two sources, and the middle party shares a source with each of the other parties.
In the network notation above, 
$M=\{1,2,3\}$, $S=\{1,2\}$ and
$$1 \leftsquigarrow 1\rightsquigarrow 2 \leftsquigarrow 2 \rightsquigarrow 3.$$
Below we provide  bounds on reduced quantum violations of some Bell inequalities of classical models for this scenario. We restrict ourselves to 2-output scenarios; note that a two-output PVM $\{P,I-P\}$ can be equivalently given by a binary observable $A$ (namely, $A=2P-I$). Thus we describe measurements of the first (resp. second; resp. third) party in terms of binary observables $A_i$ (resp. $B_i$; resp. $C_i$), as in Section \ref{sec:bell}.

\subsubsection{Example}\label{exa4}

Suppose each of the parties has two inputs and two outputs.
In \cite{Chaves}, the following inequality for classical models of this bilocal scenario is given:
\begin{equation}\label{e:sqrt}
\sqrt{|J_1|}+\sqrt{|J_2|}\le 2
\end{equation}
where
$$J_1=\lambda\big((A_1+A_2)B_1(C_1+C_2)\big),\qquad 
J_2=\lambda\big((A_1-A_2)B_2(C_1-C_2)\big).$$
The inequality \eqref{e:sqrt} is equivalent to the four polynomial inequalities
\begin{equation}\label{e:chaves}
-\frac{1}{8}(J_1-\eta_1 J_2)^2+\eta_2(\eta_1 J_1+J_2)\le 2
\end{equation}
for $\eta_i\in\{-1,1\}$; see \cite{Chaves}.
Let us bound reduced quantum violations of \eqref{e:chaves} for $\eta_1=\eta_2=1$ (the other cases are similar).
Denote
$$
j_1=\sum_{i,j\in\{1,2\}} \sig(x_iy_1z_j), \quad 
j_2=\sum_{i,j\in\{1,2\}}(-1)^{i+j} \sig(x_iy_2z_j), \quad
b = -\frac{1}{8}(j_1-j_2)^2+(j_1+j_2).
$$
The maximal bilocal reduced quantum violation of \eqref{e:chaves} is then given by the optimization problem
\begin{equation}\label{e:opt_chaves}
\begin{split}
\sup\ b\ \text{ s.t. }
&\  x_i^2=y_i^2=z_i^2=1,\quad
[x_i,y_j]=[y_i,z_j]=[z_i,x_j]=0 
\text{ for all }i,j,\\
&\ \sig\big(w_1(x_1,x_2)w_2(z_1,z_2)\big)=\sig\big(w_1(x_1,x_2)\big)\sig\big(w_2(z_1,z_2)\big)
\text{ for all }w_1,w_2.
\end{split}
\end{equation}
The SDP \eqref{eq:pure_constr_primal} for $d=3$ returns $4$, which gives an upper bound for a bilocal reduced quantum violation of \eqref{e:chaves}. This bound is attained by a spatial quantum model with
\begin{align*}
&A_1=C_1=\begin{pmatrix}
1 & 0 \\ 0 & -1
\end{pmatrix},\qquad
A_2=C_2=\begin{pmatrix}
0 & 1 \\ 1 & 0
\end{pmatrix}, \\
&B_1=\frac12
\begin{pmatrix}
 1 & 1 \\
 1 & -1 
\end{pmatrix}\otimes
\begin{pmatrix}
 1 & 1 \\
 1 & -1 
\end{pmatrix},
\quad
B_2=\frac12
\begin{pmatrix}
 1 & -1 \\
 -1 & -1
 \end{pmatrix}\otimes
 \begin{pmatrix}
 1 & -1 \\
 -1 & -1
 \end{pmatrix},\\
&\rho_j=\dyad{\psi} \quad\text{for}\quad
\ket{\psi}=
\frac{1}{\sqrt{2}}\begin{pmatrix}
 1 \\
 0 \\
 0 \\
 1
\end{pmatrix}.
\end{align*}
Also, for this bilocal model one has $\sqrt{|J_1|}+\sqrt{|J_2|}
=2\sqrt{2}$.

\subsubsection{Example}\label{i3322}

In \cite[Appendix F]{RBBPBG}, the authors prove an analog of the $I_{3322}$ inequality for classical bilocal models
\begin{equation}\label{e:sqrt1}
\sqrt{|J_1|}+\sqrt{|J_2|}\le \sqrt{|L|}
\end{equation}
where
\begin{align*}
J_1&=\frac12\lambda\big((A_1+A_2+A_3+I)B_1(C_1+C_2)\big),\\ 
J_2&=\frac12\lambda\big((A_1+A_2-A_3+I)B_2(C_1-C_2)\big)
+\frac12\lambda\big((A_1-A_2)B_3(C_1-C_2)\big), \\
L&= 4+\lambda(A_1)+\lambda(A_2).
\end{align*}
The inequality \eqref{e:sqrt1} implies the polynomial inequality
\begin{equation}\label{e:i3322}
2(J_1J_2+J_1L+J_2L)-J_1^2-J_2^2-L^2\le 0.
\end{equation}
To find an upper bound for bilocal reduced quantum violations of \eqref{e:i3322}, 
we set up the optimization problem as in the previous example, and
the SDP \eqref{eq:pure_constr_primal} for $d=3$ returns $15.6705$.
With the current computational limitations, we do not know whether this is the least upper bound. Nonetheless, there certainly exist bilocal quantum violations of \eqref{e:i3322}. Concretely, for
$\alpha=1.947$ and $\beta=1.639$, the spatial quantum model
\begin{align*}
A_1&=\begin{pmatrix}
1 & 0 \\ 0 & -1
\end{pmatrix},
\quad A_2=\begin{pmatrix}
0 & 1 \\ 1 & 0
\end{pmatrix},\quad
A_3=\begin{pmatrix}
0 & -i \\ i & 0
\end{pmatrix},\quad C_1=A_1,\quad C_2=A_2,
\\
B_1&=(\sin\alpha A_2+\cos\alpha A_3)\otimes \frac{1}{\sqrt{2}}(A_1+A_2),\quad
B_2=(\sin\alpha A_2-\cos\alpha A_3)\otimes \frac{1}{\sqrt{2}}(A_1-A_2),
\\
B_3&=-A_1\otimes \frac{1}{\sqrt{2}}(A_1-A_2),\\
\rho_j&=\dyad{\psi_j}
\quad\text{for}\quad
\ket{\psi_1}=\frac{\sin\beta}{2}
\begin{pmatrix}
\sqrt{2} \\ -1 \\ 0 \\ -1
\end{pmatrix}
+\frac{\cos\beta}{2}
\begin{pmatrix}
0 \\ -1 \\ \sqrt{2} \\ 1
\end{pmatrix},
\quad
\ket{\psi_2}=\frac{1}{\sqrt{2}}\begin{pmatrix}
 0 \\ 1 \\ -1 \\ 0
\end{pmatrix}
\end{align*}
gives $2(J_1J_2+J_1L+J_2L)-J_1^2-J_2^2-L^2
= 13.3309$.

In the above examples, state polynomial optimization yields the same result both in the real and the complex framework (though complex SDP relaxations are larger). In the following example, the real and the complex frameworks yield slightly different results.

\subsubsection{Example}\label{exaZ0}

Suppose that each of the parties has three inputs and two outputs.
In \cite{tavakoli22}, an inequality for classical models is given:
$$\frac13 S-T \le 3+5Z$$
where
\begin{align*}
S&=\sum_{i\in\{1,2,3\} } \Big(\lambda(B_iC_i)-\lambda(A_iB_i)\Big) ,\\
T&=\sum_{\{i,j,k\}=\{1,2,3\} } \lambda(A_iB_jC_k), \\
Z&=\max\Big(
\{
|\lambda(A_i)|,|\lambda(B_i)|,|\lambda(C_i)|\colon i\in\{1,2,3\}\} \\
&\phantom{=\max\Big(x}\cup 
\{
|\lambda(A_iB_j)|,|\lambda(B_iC_j)|,|\lambda(A_iC_j)|\colon i\neq j\} \\
&\phantom{=\max\Big(x}\cup 
\{
|\lambda(A_iB_jC_k)|\colon |\{i,j,k\}|
\le2 \}
\Big).
\end{align*}
In particular, \cite{tavakoli22} focused on the inequality
\begin{equation}\label{exaZ0:eq}
    \frac13 S-T \le 3 \qquad \text{subject to }Z=0
\end{equation}
for classical models, and showed it admits a spatial quantum violation satisfying $Z=0$ and $\frac13 S-T=4$. 
To provide an upper bound for bilocal reduced quantum violations of \eqref{exaZ0:eq}, 
we set up the optimization problem as before. The SDP \eqref{eq:pure_constr_primal} for $d=3$ (more precisely, its reduction as in Section \ref{sec:ss} and Section \ref{sec:uni}) contains four PSD blocks with respective size $130, 105, 105, 105$ and $3018$ affine constraints. We obtain an upper bound $4.46613$ in 1.94s.
For $d=4$, the SDP \eqref{eq:pure_constr_primal} (more precisely, its reduction as in Section \ref{sec:ss} and Section \ref{sec:uni}) contains four PSD blocks with respective size $678$, $678$, $678$, $646$ and $64878$ affine constraints. We obtain an upper bound $4.37666$ for the reduced quantum violations of \eqref{exaZ0:eq} in 5346s. The corresponding Hankel matrix is not flat and so we cannot certify the optimality of this bound. The SDP \eqref{eq:pure_constr_primal} for $d=5$ contains four PSD blocks with respective size $3838$, $3838$, $3838$, $3739$ and $1\,352\,093$ affine constraints. With {\tt COSMO} as an SDP solver, we obtain the upper bound $4.36605$ in $566\,979$s.

Now we turn to the results obtained in the complex setting. For $d=3$, the SDP \eqref{e:cx_dual} contains four PSD blocks with respective size $178, 134, 134, 134$ and $7578$ affine constraints. We obtain an upper bound $4.46665$ in 20.1s. For $d=4$, the SDP \eqref{e:cx_dual} contains four PSD blocks with respective size $1012, 1012, 1012, 958$ and $210\,501$ affine constraints. We obtain an upper bound $4.37951$ in 47138s. For $d=5$, the SDP \eqref{e:cx_dual} contains four PSD blocks with respective size $6498, 6498, 6498, 6244$ and $5\,672\,003$ affine constraints, which is currently out of reach.

\subsubsection*{\bf Conflict of interest}

The authors declared that they have no conflict of interest.

\bibliographystyle{alpha}
\bibliography{statencpop}
\end{document}